 \documentclass[reqno,11pt,a4paper]{amsart}

\usepackage{pgfplots}
\usepackage{subfigure}
\pgfplotsset{compat=1.12} 
\usepackage{listings}

\usepackage{amsfonts,amsmath,amssymb,amsthm}
\usepackage[latin1]{inputenc}
\usepackage{dsfont}
\usepackage{enumerate}
\usepackage{xcolor}
\usepackage[all]{xy}
\def\notshow#1\notshowend{} %
\usepackage{hyphenat}
\usepackage{microtype}

\usepackage{graphicx}

\usepackage{tikz}
\usepackage{tikz-3dplot}
\usepackage{pgfplots}

\usepackage{multirow}
\usepackage{array}

\newcommand{\df}{\mathrm{d}}
\usepackage{amssymb}
\usepackage{amsfonts}
\usepackage{mathrsfs}
\usepackage{hyperref}

\usepackage[normalem]{ulem} 

\def\bb#1\eb{\textcolor{blue}{#1}} 
\def\br#1\er{\textcolor{red}{#1}} %
\def\bm#1\em{\textcolor{purple}{#1}} %

\newcommand{\R}{\mathds R}

\newtheorem{thm}{Theorem}[section]
\newtheorem{prop}[thm]{Proposition}

\theoremstyle{definition}
\newtheorem{defi}[thm]{Definition}

\newtheorem{rem}[thm]{Remark}

\newtheorem*{theorem*}{Proposition}

\numberwithin{equation}{section}

\newcommand{\ben}{\begin{enumerate}}
\newcommand{\een}{\end{enumerate}}
\newcommand{\bit}{\begin{itemize}}
\newcommand{\eit}{\end{itemize}}
\newcommand{\edoc}{\end{document}}

\title[General model for wildfire propagation]{A general model for wildfire propagation with wind and slope}

\author[M. A. Javaloyes]{Miguel \'Angel Javaloyes}\address{Departamento de Matem\'{a}ticas, \hfill\break\indent Universidad de Murcia, \hfill\break\indent Campus de Espinardo,\hfill\break\indent 30100 Espinardo, Murcia, Spain} \email{majava@um.es}

\author[E. Pend\'{a}s-Recondo]{Enrique Pend\'{a}s-Recondo}\address{Departamento de Matem\'{a}ticas, \hfill\break\indent Universidad de Murcia, \hfill\break\indent Campus de Espinardo,\hfill\break\indent 30100 Espinardo, Murcia, Spain \hfill\break\indent \& \hfill\break\indent IMAG (Centro de Excelencia Mar\'{i}a de Maeztu), \hfill\break\indent Universidad de Granada, \hfill\break\indent 18071 Granada, Spain}\email{e.pendasrecondo@um.es}

\author[M. S\'{a}nchez]{Miguel S\'{a}nchez}\address{Departamento de Geometr\'{\i}a y Topolog\'{\i}a, Facultad de Ciencias \& \hfill\break\indent IMAG (Centro de Excelencia Mar\'{i}a de Maeztu), \hfill\break\indent Universidad de Granada, \hfill\break\indent 18071 Granada, Spain}\email{sanchezm@ugr.es}


\begin{document}
\begin{abstract}
A geometric model for the computation of the firefront of a forest wildfire which takes into account several effects (possibly time-dependent wind, anisotropies, and slope of the ground) is introduced. It relies on a general theoretical framework, which reduces the hyperbolic PDE system of any wave to an ODE in a Lorentz-Finsler framework. The wind induces  a sort of  double semi-elliptical fire growth, while the influence of the slope is modeled by means of a term which comes from the Matsumoto metric (i.e., the standard non-reversible Finsler metric that measures the time when going up and down a hill). These contributions make a significant difference from previous models because, now, the infinitesimal wavefronts are not restricted to be elliptical. Even though this is a technical complication, the wavefronts remain computable in real time.
Some simulations of evolution are shown, paying special attention to possible crossovers of the fire.

\vspace{5mm}

\end{abstract}
\maketitle


\tableofcontents
\section{Introduction}

Predicting the spread of wildfires is a noteworthy application of mathematical modeling currently used by enviromental agencies worldwide. Examples of widely used fire behaviour simulators are BehavePlus \cite{Andr}, FARSITE \cite{F} (which uses the National Fire-Danger Rating System \cite{CD}), FlamMap \cite{F2} (which includes FARSITE), the Fuel Characteristic Classification System \cite{OSRP}, the KITRAL System \cite{PJ} and Prometheus \cite{TBWTA} (which uses the Canadian Forest Fire Danger Rating System \cite{S}). Among the contributions to the models, we stress:
\begin{itemize}
\item Anderson (1983) \cite{An} proposed an elliptical propagation under constant conditions of wind. More precisely, a sort of {\em double semi-ellipse} is pointed out as the best experimental fitting, even though the elliptical propagation yields a good approximation and subsequent models use it. This means that the velocity of the fire at each point is given by an ellipse, which is chosen in such a way that its major axis is aligned with the wind direction and its length-to-width ratio depends only on the wind speed \cite{Al,An}.

\item Using Huygens' principle, i.e., that each point of the firefront is the source of an independent point-ignited fire, Richards (1990) \cite{R} stated the PDE system for a general infinitesimally elliptical propagation. 

\item These equations provide the fire growth of the main {\em surface fire}, but a variety of secondary effects also occur during a wildfire, such as the {\em crown fire} at the top of the trees and the {\em spotting} (new ignition points ahead of the main firefront) caused by small particles drawn by the wind. Several physical models have been proposed in order to describe each aspect of the fire behavior individually. A selection of these models were incorporated by Finney (1998) \cite{F} into a unified one, leading to the simulator FARSITE.\footnote{Among all the fire growth simulators that have been developed, FARSITE is of greatest interest for us because its approach based on Huygens' principle and Richards' equations is the one we adopt here and in \cite{JPS}. Prometheus also follows the same approach.}

\item Being Richards' equations completely general, some choices have to be made in order to build a concrete model. One of the most important is the position of the ignition point, which is assumed by FARSITE to be one of the foci of the ellipse. Although not clear from an experimental viewpoint, this is an approach that has been applied before and greatly simplifies the construction of the ellipses (see the discussion and references in \cite{Al}).
 

\item More recently, Markvorsen (2012) \cite{M12} introduced the idea of using Finsler geometry for the modeling of wildfires, showing also the importance of the cut points. This Finslerian viewpoint was further developed in \cite{M16} in order to transform  Richards' PDE into an ODE, which is interpreted as the geodesic equations for a Finsler metric of Randers type. He also introduced a rheonomic Lagrangian viewpoint \cite{M17} to include the possible dependence of the wind and the fuel on the time. Further recent developments can be seen in \cite{D}.
\end{itemize} 

Recently, we have introduced a general model for wave propagation that provides a full geometric picture of the evolution of its wavefront \cite{JPS}. This combines a spacetime viewpoint (in the spirit of the modeling of light propagation at a finite speed in General Relativity; see e.g. \cite{BLV05,Gib,JS20}) with  Markvorsen's one \cite{M16} on the use of Finsler geometry to model anisotropies in the speed of propagation. The fastest propagation curves appear as the lightlike geodesics for a Lorentz-Finsler metric, which obey a simple ODE system.
Here we will apply this approach to forest wildfires (focusing only on the surface fire), obtaining an efficient computation of their firefronts in terms of these geodesics. Moreover, the cut points of such geodesics allow one to determine the crossovers of the fire, which are especially important as they may become a danger for firefighters.

The anisotropies we will consider go beyond the elliptical ones in previous approaches. Technically, the assumption that the infinitesimal wavefronts are ellipses (thus, determined by the position of their centers and the orientation and length of their axes) can be regarded as a first quadratic approximation of the model. This leads to modeling the velocity of the fire by terms of a (possibly $ t $-dependent) Randers metric, whose indicatrices are (non-centered) ellipses. However, this approximation might lead to a non-realistic evolution over time. Our aim here will be to build a more accurate model: the effect of the wind will produce a pattern similar to a double semi-ellipse, while the effect of the slope will be described by using a Matsumoto-type metric. This is a classical Finsler metric that allows one to characterize the fastest trajectories to go up and down a mountain or a hill \cite{SS}. In the case of wildfires, the effect of the metric is reversed, as the upward propagation is faster than the downward one. Besides, as we will see, the elliptical approximation is not necessary from a computational viewpoint. Indeed, the ODE system satisfied by our non-quadratic model becomes relatively simple for a big range of possibilities and it can be solved in a short time of computation.

The paper is structured as follows. In \S~\ref{sec:general_setting}, the usual definitions regarding Minkowski norms and Finsler metrics are introduced, along with the mathematical setting upon which we will construct our model. Finsler metrics enable us to model the anisotropic propagation of a wave. After taking aerial coordinates, this propagation takes place on a space $ N\subset  \R^2 $, but in order to properly describe a possible time-dependence we will work in a spacetime $ M=\R \times N $, in the spirit of our approach for waves \cite{JPS}. 

In \S~\ref{sec:matsumoto} we explicitly construct the model. The goal is to obtain a Finsler metric whose indicatrix coincides with the infinitesimal propagation (unit-time propagation assuming constant conditions) of the real wildfire. The construction of such a metric is based on very simple assumptions. As it is purely theoretical,  empirical data will be needed to check its validity, improve the initial assumptions and incorporate additional effects. Anyway, it can serve as a good starting point to implement this kind of geometric models to realistic cases. First, we consider in \S~\ref{subsec:no_wind} the case without wind, only taking into account the effect of the slope. The model relies on three terms: the first two are angle-independent contributions (i.e., the same in all directions) and provide the wildfire propagation if there was no slope, while the last one implements the contribution of the slope in terms of the angle between the vertical directon and the sloped ground. In this way, one arrives at a Matsumoto-type metric, thus providing a link with the distance covered by a walker on a slope (Rem.~\ref{rem:matsumoto}). Next, in \S~\ref{subsec:wind} the wind is included and the first spherical (angle-independent) term in the previous model now becomes elliptical. This generates, along with the second spherical term, a shape that resembles that of a double semi-ellipse aligned with the wind. Finally, the slope contribution remains the same. Under certain simple conditions, this becomes a Finsler metric and the corresponding geodesic equations can be stated easily. 

In \S~\ref{sec:calculation} we summarize and adapt the main results developed in \cite{JPS}. First, we recall in \S~\ref{ss_spacetime_viewpoint} the spacetime framework and define the {\em firemap}, which tracks the wildfire propagation. Then we derive in \S~\ref{ss_ode} the ODE system for lightlike geodesics that provide the firefront at each instant of time in our model (Thm.~\ref{th:pde_ode}); the alternative PDE system (equivalent to Richards' equations if the propagation was elliptical) is also obtained for comparison.

Finally, \S~\ref{sec:examples} shows some examples of our model, obtained by computationally solving the aforementioned ODE system for the firefront. Three basic features are shown:

\begin{itemize}
\item[\S~\ref{subsec:slope_wind}]
{\em Effect of the slope vs wind}. Usually, 
the effect of the slope is simplified and assumed to be of the same nature as that of the wind, so that both yield the same kind of changes in the wildfire propagation (see, e.g., \cite{F}, which defines a ``virtual wind'' that by itself produces the combined effect of both phenomena). This is not clear from an experimental viewpoint and, in our model, both effects have a qualitative difference.

\item[\S~\ref{subsec:cut_points}] {\em Cut points and fire crossovers}. 
Computationally speaking, cut points represent a problem from the PDE viewpoint. Indeed, overlaps of burned regions start to appear after the first cut point. Beyond it, the PDE system may not provide any solution, or if it does, then it no longer represents the actual firefront, as these equations do not intrinsically distinguish burned from unburned areas. Thus, the total burned area and its perimeter must be corrected in order to properly continue the computation. Several algorithms have been proposed (see \cite{F} and references therein), but the process is usually expensive in time and computing power. 

However, our example therein illustrates that this problem is greatly simplified when solving the ODE system. Indeed, instead of the firefront, what we obtain now are the trajectories of the fire at each time. Those entering an already burned region provide directly a cut point and they can be simply removed with no harm to the computation of the others. This is a major goal of the ODE viewpoint, because of the extreme importance of crossovers for firefighters. 

\item[\S~\ref{subsec:real_wildfires}] {\em Flexibility}. Our model can be adapted to a wide range of cases. An example of time-dependent wind and a spatial variation in the fuel conditions is shown, along with another case with a more complex topography.
\end{itemize}

Lastly, we discuss in Appendix~\ref{appendix} two technical issues. The first one in Appendix~\ref{appendix1} is a characterization of the cut points for lightlike geodesics. This is an important issue, because its real-time detection during wildfires is crucial. Indeed, regions between intersecting trajectories may be completely surrounded by the fire and, therefore, may become extremely dangerous when attempting to extinguish the fire. The second one in Appendix~\ref{appendix2} is a discussion on the basic assumption of strong convexity for Finsler metrics. When the metric fails to be strongly convex, geodesics cannot be obtained directly in some directions, so some modifications are proposed in order to avoid this problem.

\section{General setting}
\label{sec:general_setting}
Minkowski norms are obtained from classical norms by relaxing the symmetry of the latter with respect to origin but strenghtening their regularity and non-degeneracy. Then, Finsler metrics are obtained from classical Riemannian ones by replacing the pointwise scalar products of the latter by Minkowski norms. More precisely:

\begin{defi}
\label{def:finsler}
Let $ N $ be a smooth (namely, $C^\infty$) manifold and $ V $ a real vector space of finite dimension. 
\begin{enumerate}[(i)]
\item A {\em Minkowski norm} on $ V $ is a continuous non-negative function $ F: V \rightarrow [0,+\infty) $ such that
\begin{enumerate}
\item $ F $ is smooth and positive on $ V \setminus \{0\} $,
\item $ F $ is positive one-homogeneous, i.e., $ F(\lambda v)=\lambda F(v) $ for all $ \lambda>0, v\in V $, and
\item for every $ v\in V\setminus\{0\} $, the {\em fundamental tensor} $ g_v^F $, defined as
\begin{equation}
\label{eq:fund_tensor}
g_v^F(u,w) := \left. \frac{1}{2}\frac{\partial^2}{\partial\delta\partial\eta}F(v+\delta u+\eta w)^2 \right\rvert_{\delta=\eta=0}, \quad \forall u,w \in V,
\end{equation}
is positive definite.\footnote{This is equivalent to saying that the indicatrix $ \Sigma $ (unit vectors) of $ F $ is strongly convex, i.e., the second fundamental form of $ \Sigma $ as an affine hypersurface of $ V $ with respect to some, and then any, transverse vector pointing into the interior of $ \Sigma $ is positive definite (see \cite[Prop. 2.3]{JS14}).}
\end{enumerate}

\item A {\em Finsler metric} on $ N $ is a continuous non-negative function $ F: TN \rightarrow [0,\infty) $, where $ TN $ is the tangent bundle of $ N $, such that it is smooth away from the zero section and each $ F_p:=F|_{T_pN} $ is a Minkowski norm, for all $ p\in N $. This definition can be naturally extended to any subbundle of $ TN $: the Finsler metric in this case becomes a smooth distribution of Minkowski norms in each fibre of the vector bundle.

\item We say that $ v \in T_pN $ is {\em $ F $-orthogonal} to $ u\in T_pN $, denoted $ v \bot_F u $, if
\begin{equation*}
g_v^F(v,u) = \left. \frac{1}{2}\frac{\partial}{\partial\delta}F(v+\delta u)^2 \right\rvert_{\delta=0}=0.
\end{equation*}
\end{enumerate}
\end{defi}

Observe that the fundamental tensor \eqref{eq:fund_tensor} can be seen as a scalar product for each direction $ v\in T_pN $. This feature will allow us to model the anisotropic propagation of a wildfire, as detailed below. 

Let $ \hat N \subset \R^3 $ be the surface where the wildfire takes place, $ (x,y,z) $ the natural coordinate functions on $ \R^3 $ and $ N\subset\R^2 $ the projection of $ \hat N $ on the $ xy $-plane, i.e., $ N $ is what one would see from an aerial view. In other words, we are selecting a (global) coordinate chart $ (\hat N,\hat z^{-1}) $, where $ \hat{z} $ is a graph:
\begin{equation*}
\begin{array}{cccc}
\hat{z} \colon & N \subset \R^2 & \longrightarrow & \hat N \subset \R^3\\
& (x,y) & \longmapsto & \hat{z}(x,y) := (x,y,z(x,y)).
\end{array}
\end{equation*}
These coordinates $ (x,y) $ will be called the {\em aerial coordinates} of $ \hat N $. In the model we will mainly work in these coordinates, passing all the information from the actual surface $ \hat N $ to $ N $ via $ \hat z $ in order to compute the wildfire propagation more easily, and then back to $ \hat N $ to see the actual spread.

In order to include the (non-relativistic) time $ t $ in the model, we define the {\em spacetime} $ M:=\R\times N $, being $ t: M \rightarrow \R $ the natural projection. At each point $ (t,p) \in M $, the fire spreads over $ N $ in every direction, although in general we will allow the velocity to vary from one direction to another. This way, the space of velocities of the fire at $ (t,p) \in M $ will be given by a closed curve (diffeomorphic to a sphere) $ \Sigma_{(t,p)} $ on the vector space $ \text{Ker}(dt_{(t,p)})$.\footnote{Note that $ \text{Ker}(dt_{(t,p)}) = T_{(t,p)}(\lbrace t \rbrace \times N) \equiv T_pN $, i.e., $ \text{Ker}(dt) $ is composed of copies of $ T_pN $ at different times.}
This means that a vector $ v \in \textup{Ker}(dt_{(t,p)}) $ represents the velocity of the firefront in the direction defined by $ v \in T_pN $ at the time $ t \in \R $ if and only if $ v \in \Sigma_{(t,p)} $ (see Fig.~\ref{fig:surface}).

\begin{figure}
\centering
\includegraphics[width=1\textwidth]{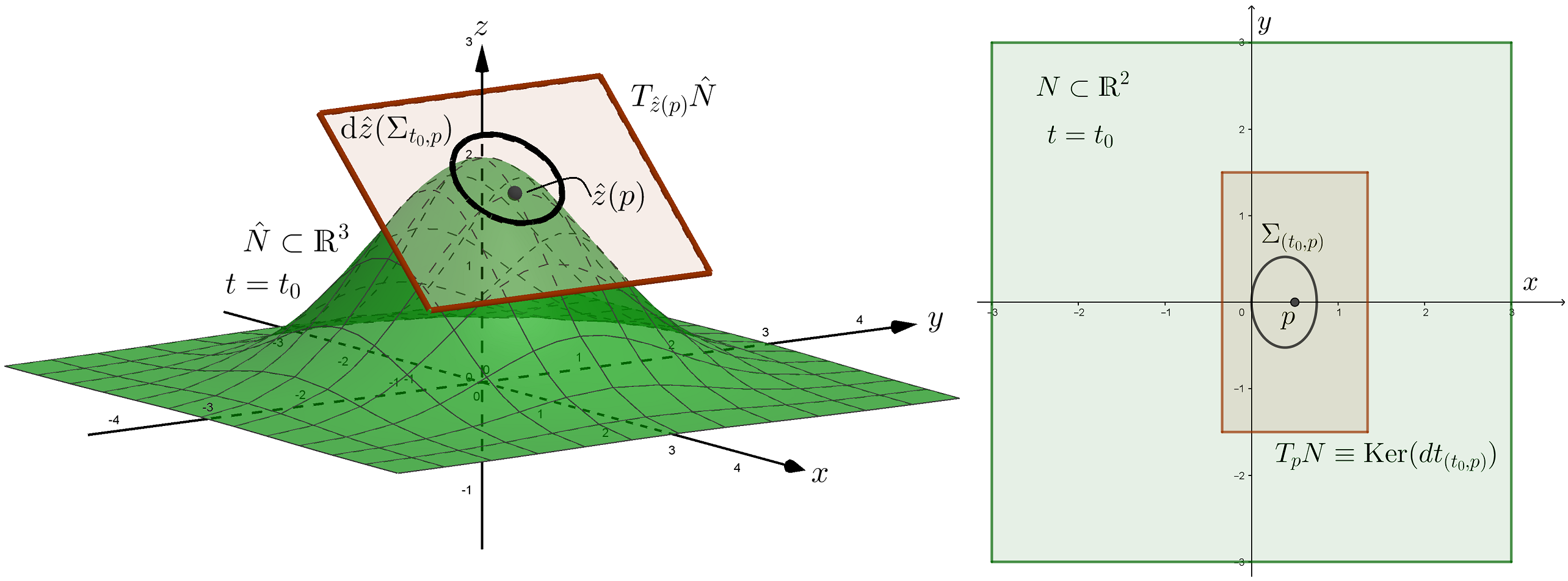}
\caption{The wildfire takes place on a surface $ \hat N \subset \R^3 $, although we will work in $ N \subset \R^2 $ through the aerial coordinates and then recover the actual information via $ \hat z: N\rightarrow \hat N $ and $ \df \hat z_p:T_pN \rightarrow T_{\hat z(p)}\hat N $. The indicatrix $ \Sigma_{(t_0,p)} $ provides the velocity of the fire (or more accurately, the projection of the actual velocity) for each direction at $ (t_0,p) \in M =\R\times N $. In the left image, the surface is given by $ z(x,y)=2\exp(-x^2/2-y^2/2) $ and the indicatrix is computed using the model developed in \S~\ref{subsec:no_wind}.}
\label{fig:surface}
\end{figure}

We can assume that $ \Sigma_{(t,p)} $ is strongly convex (an oval) and encloses the zero vector, which  are natural conditions when dealing with the propagation of waves in general. Then $ \Sigma_{(t,p)} $ is the indicatrix (i.e., the unit vectors) of a Minkowski norm $ F_{(t,p)} $ on $ \textup{Ker}(dt_{(t,p)}) $ (see \cite[Thm. 2.14]{JS14}). Conversely, the indicatrix $ \Sigma $ of any Finsler metric $ F $ on the vector bundle $ \text{Ker}(dt) \subset TM $ (or, equivalently, a time-dependent Finsler metric on $ TN $) provides a field of such ovals $ \Sigma_{(t,p)} $ varying smoothly with $ (t,p) \in M $, and therefore they can be regarded as the space of initial velocities of a wildfire ignited at the point corresponding to the origin of the tangent space. 

In general, we will assume that the wildfire starts from a simple closed curve $ S_0 $ in $ \{t=0\} $,\footnote{Throughout this work, we use the notation $ \{t=t_0\}:=\{t_0\}\times N $.} which is the boundary of the initial burned area $ B_0 \subset M $. Then, given $ S_0 $ and $ F $, i.e., given the initial firefront and the velocity of the fire varying in space, direction and time, \cite{JPS} develops the ODE system that allows one to calculate the firefront at each instant of time (see also \cite{M16,M17}). So, our aim in this work is to find the metric $ F $ that best matches the velocities of the fire.

\begin{rem}
\label{rem:infinit_fire}
In practice, the choice of the oval $ \Sigma_{(t,p)} $ at each point will depend on the amount of solid and gaseous fuel, the slope of the surface and meteorological conditions such as the wind, the temperature, the amount of oxygen in the air, etc. (see, e.g., \cite{AFS,Rot}). In order for the model to be as realistic as possible, we must allow every parameter to vary in space and time, i.e., they may depend on $ p $ and $ t $. At the end, the indicatrix $ \Sigma $ of $ F $ should precisely model the infinitesimal firefront, i.e., $ \Sigma_{(t,p)} $ is what one would obtain after one time unit if the fire was spreading over a surface where every parameter (slope, wind, etc.) is constant.
\end{rem}

\section{Non-elliptical Finslerian model for forest wildfires}
\label{sec:matsumoto}
The simplest analytical anisotropic approximation to model the spread of the fire is to assume that the space of velocities $ \Sigma $ is an ellipse at each point. Anderson \cite{An} was the first to check the validity of this approximation for wind-driven wildfires and thenceforth, elliptical propagation has been commonly assumed in fire growth models and even extended to the case when the wildfire takes place on a sloped terrain (see, e.g., \cite{F}).

In this section we present a more sophisticated model for the propagation of the fire, which may serve as a starting naive modelization for an even more complicated one based on experimental data. Among other improvements, we stress that this model removes the quadratic restriction (i.e., elliptical propagation), with the effect of the slope and the wind being qualitatively different. Specifically: (a) with slope, the contribution to the otherwise spherical propagation is purely geometrical, obtaining a (reverse) Matsumoto metric, and (b) with wind, the infinitesimal propagation takes the form of a double semi-ellipse, which seems to be the most realistic representation from an experimental viewpoint \cite{An}.

\subsection{Simple model with slope}
\label{subsec:no_wind}
First we describe the case when there is no wind. Recall from \S~\ref{sec:general_setting} that the wildfire spreads over a surface $ \hat N \subset \R^3 $, which is the image of a graph $ \hat{z}(x,y) = (x,y,z(x,y)) $, with $ (x,y) \in N \subset \R^2 $. We will denote by $ \langle \cdot,\cdot \rangle $ and $ || \cdot || $ the Euclidean metric in $ \R^3 $ and its corresponding norm, respectively. Let us study the speed of the fire $ s_{fire} $ as a positive function on $ \R \times N \times \mathds{S}^1 $. Namely, given a time $ t \in \R $, a point $ p = (x,y) \in N $ and an (oriented) direction $ \theta \in [0,2\pi) \cong \mathds{S}^1 $ at $p$, we propose the following speed of the fire:
\begin{equation}
\label{eq:mod_no_wind}
s_{fire}(t,p,\theta) = a(t,p) + h(t,p)[1+\cos(\delta(p,\theta))].
\end{equation}
Let us detail the physical meaning of each term of the equation:
\begin{itemize}
\item $ a(t,p) $ and $ h(t,p) $ are positive real functions on $ M=\R\times N $ that represent different (although possibly related) fire contributions: $ h(t,p) $ refers essentially to that of the fire flame, while the sum $ a(t,p)+h(t,p) $ must coincide with the total speed of the fire when it spreads over a plane without slope ($ \delta = \pi/2 $). Fixing an instant of time, both functions can vary due to the change of vegetation and soil conditions from one point to another; fixing a point, their variation depend on that of the temperature, moisture, rain, etc. Therefore, $ a(t,p) $ and $ h(t,p) $ may vary over space and time, reflecting non-homogeneous and time-dependent fuel and meteorological conditions.

\item $ \delta(p,\theta) \in (0,\pi) $ is the (vertical) {\em slope angle}, that is, the angle in $\R^3$ between the $z$-axis and the surface $\hat N$ in the  chosen aerial direction $ \theta $, 
i.e.,
\begin{equation}
\label{eq:delta}
\cos(\delta(p,\theta)) = \frac{\langle \df\hat z_p(\cos\theta,\sin\theta),\partial_z|_p \rangle}{||\df\hat z_p(\cos\theta,\sin\theta)||}.
\end{equation}
This way, the term $ h\cos\delta $ is the projection of the ``flame" on the tangent space to the surface and represents the contribution of the slope to the fire spread. Physically, the fire moves faster upwards than downwards, since the (vertical) flame is closer to the ground in the upward direction. Here we model this fact through the vertical vector $ h\partial_z $, whose contribution is greater moving upwards, i.e., when $ \delta $ is smaller. Clearly,  $ \delta $ does not depend on $t$ as neither does the slope.
\end{itemize}

Note that the term $ 1+\cos\delta $ ensures that the contribution of the flame is always positive and therefore, so is the speed of the fire for every direction, no matter the conditions.

\subsubsection{Emergence of the Matsumoto metric}
Given any $ \theta \in [0,2\pi) $, $ u_{\theta}|_p $ will denote the corresponding unit vector in $ T_{\hat z(p)}\hat N $, i.e.,
\begin{equation}
\label{eq:u}
\begin{split}
u_{\theta}|_p := \ & \frac{\df\hat z_p(\cos\theta,\sin\theta)}{||\df\hat z_p(\cos\theta,\sin\theta)||} = \frac{(\cos\theta,\sin\theta,\df z_p(\cos\theta,\sin\theta))}{\sqrt{1+\df z_p(\cos\theta,\sin\theta)^2}} = \\
= \ & \frac{(\cos\theta,\sin\theta,\cos\theta\partial_xz|_p+\sin\theta\partial_yz|_p)}{\sqrt{1+(\cos\theta\partial_xz|_p+\sin\theta\partial_yz|_p)^2}},
\end{split}
\end{equation}
where we have used  that, for any $ v = v_1\partial_x|_p + v_2 \partial_y|_p = (v_1,v_2) \in T_pN $,
\begin{equation*}
\begin{split}
& \df \hat{z}_p(v) = (v_1,v_2,\df z_p(v)), \\
& \df z_p(v) = v_1\frac{\partial z}{\partial x}(p)+v_2\frac{\partial z}{\partial y}(p) = v_1\partial_xz|_p+v_2\partial_yz|_p.
\end{split}
\end{equation*}
This way, $ v_{fire}(t,p,\theta):= s_{fire}(t,p,\theta)u_{\theta}|_p \in T_{\hat z(p)}\hat N $ represents the actual velocity of the fire (as a vector) at $ (t,p) $ in the direction $ \theta $.
Conversely, given any $ v=(v_1,v_2) \in T_pN $, $ \theta_v $ will denote its angular direction, determined by $ \tan\theta_v = v_2/v_1 $.

Using \eqref{eq:delta} and the notation above, \eqref{eq:mod_no_wind} becomes
\begin{equation*}
s_{fire}(t,p,\theta) = a(t,p) + h(t,p) (1+\langle u_\theta|_p,\partial_z|_p \rangle),
\end{equation*}
or, multiplying by $ s_{fire} $ on both sides:
\begin{equation}\label{eq:no_delta}
\begin{split}
s_{fire}(t,p,\theta)^2 = \ & a(t,p)s_{fire}(t,p,\theta) +\\
& h(t,p)(s_{fire}(t,p,\theta)+\langle s_{fire}(t,p,\theta) u_\theta|_p,\partial_z|_p \rangle).
\end{split}
\end{equation}

Our aim is to construct a Finsler metric $ F $ on $ \textup{Ker}(dt) $ (or equivalently, on $ TN $ but time-dependent) whose unit vectors represent the aerial velocities of the fire, that is,
\begin{equation*}
\begin{split}
F_{(t,p)}(v) = 1 & \Leftrightarrow \df \hat{z}_p(v) = s_{fire}(t,p,\theta_v) u_{\theta_v}|_p \Leftrightarrow ||\df \hat{z}_p(v)|| = s_{fire}(t,p,\theta_v) \Leftrightarrow \\
\Leftrightarrow \alpha_p(v)^2 & = a(t,p)\alpha_p(v) + h(t,p) (\alpha_p(v)+\langle \df\hat z_p(v),\partial_z|_p \rangle) =\\ 
& = a(t,p)\alpha_p(v) + h(t,p)(\alpha_p(v)+\beta_p(v)),
\end{split}
\end{equation*}
where in the last equivalence we have used \eqref{eq:no_delta} and the notation
\begin{equation}
\label{eq:alpha_beta}
\begin{split}
& \beta_p(v) := \df z_p(v) = v_1\partial_xz|_p + v_2\partial_yz|_p, \\
& \alpha_p(v) := ||\df \hat{z}_p(v)|| = \sqrt{v_1^2+v_2^2+\beta_p(v)^2},
\end{split}
\end{equation}
for any $ v = (v_1,v_2) \in T_pN $.

\begin{figure}
\centering
\includegraphics[width=1\textwidth]{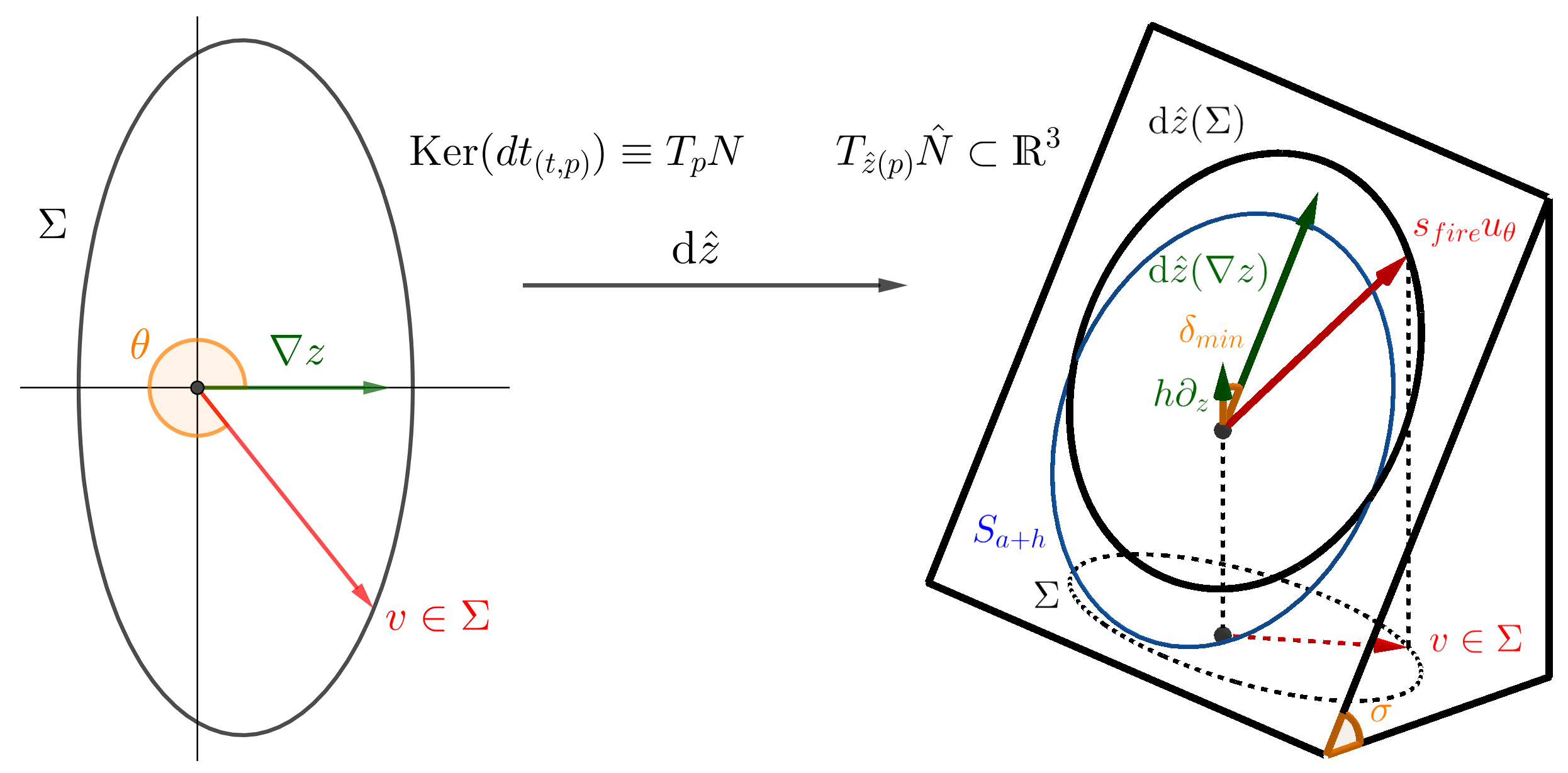}
\caption{For each direction $ \theta $, $ \Sigma $ provides the aerial velocity of the fire $ v \in \Sigma $, being the actual velocity $ \df \hat z(v)=s_{fire}u_{\theta} \in \df \hat z(\Sigma) $. The sphere $ S_{a+h} $ of radius $ a+h $ would be the indicatrix without slope, and it is depicted here in order to appreciate how the contribution of the slope shifts it. The vector $ \df \hat{z}(\nabla z) $ points to the direction of maximum slope or, equivalently, minimum $ \delta $, so that $ \sigma = \frac{\pi}{2}-\delta_{min} $, being $ \sigma $ the angle of inclination. Data: $ \partial_xz = \sqrt{3} $, $ \partial_yz = 0 $ (so that $ \sigma = \frac{\pi}{3} $), $ a = 2 $ and $ h = 1 $.}
\label{fig:slope}
\end{figure}

Therefore, the indicatrix $ \Sigma $ of $ F $ at a time $ t \in \R $ and a point $ p \in N $ must be (see Fig.~\ref{fig:slope})
\begin{equation}
\label{eq:indx_matsumoto}
\Sigma_{(t,p)} = \lbrace v \in \textup{Ker}(dt_{(t,p)}): Q_{(t,p)}(v) = 0 \rbrace,
\end{equation}
where $Q_{(t,p)}(v) := \alpha_p(v)^2-a(t,p)\alpha_p(v)-h(t,p)(\alpha_p(v)+\beta_p(v))$. Applying Okubo's technique (see, e.g., \cite[Exercise 1.2.8(a)]{BCS}), the candidate Minkowski norm $ F_{(t,p)} $ we are looking for is characterized by the equation $ Q_{(t,p)}(\frac{v}{F_{(t,p)}(v)}) = 0 $. Solving for $ F_{(t,p)}(v) $ we obtain
\begin{equation}
\label{eq:F_p_no_wind}
F_{(t,p)}(v) = \frac{\alpha_p(v)^2}{a(t,p)\alpha_p(v)+h(t,p)(\alpha_p(v)+\beta_p(v))},
\end{equation}
for all  $v \in \textup{Ker}(dt_{(t,p)})\setminus\{0\}$. Fig.~\ref{fig:indicatrices} depicts the indicatrices of $ F $ in a specific example.

\begin{figure}
\centering
\includegraphics[width=0.8\textwidth]{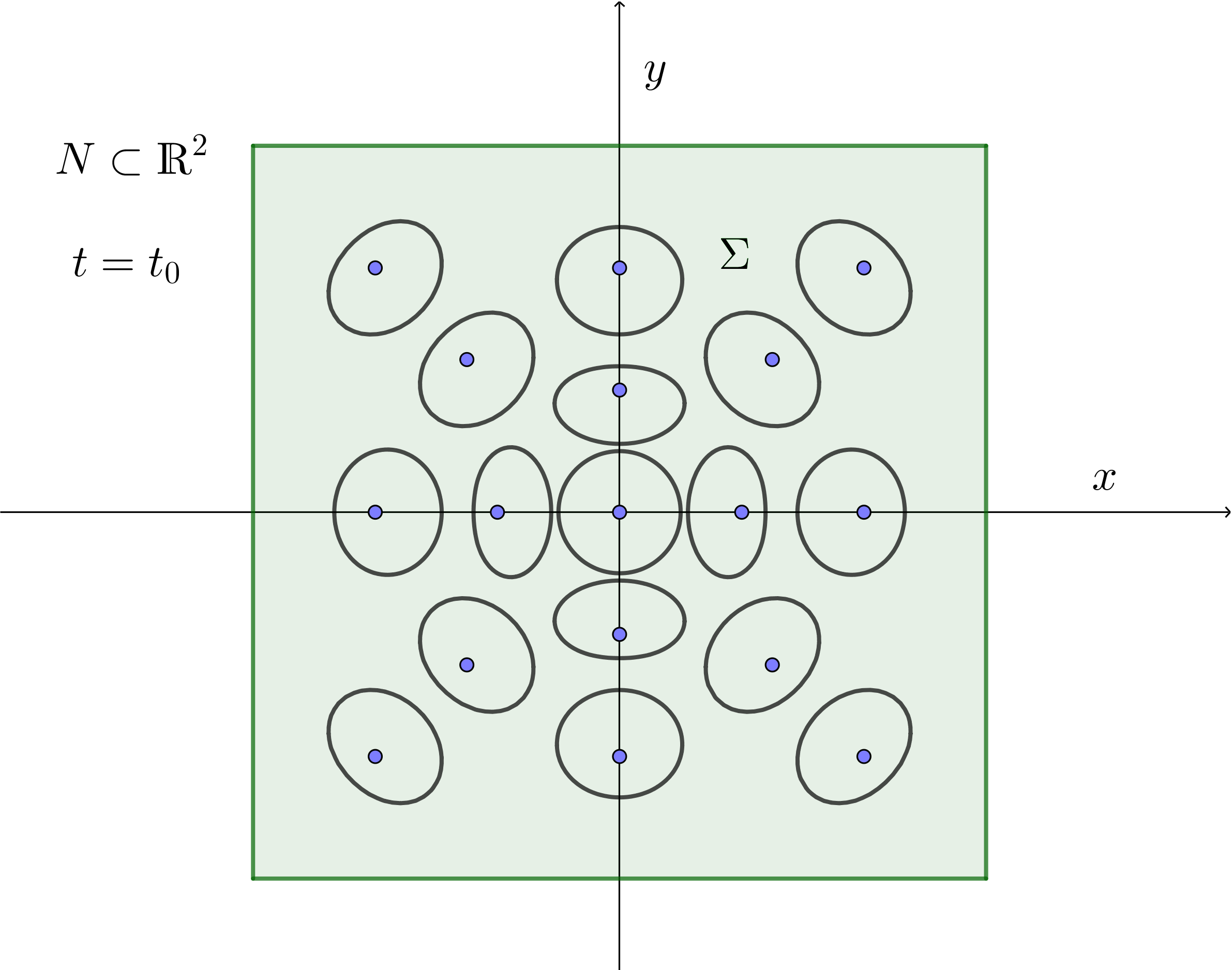}
\caption{Indicatrices of the Matsumoto metric given by \eqref{eq:F_p_no_wind} to model the wildfire spread on the surface of Fig.~\ref{fig:surface}. Data: $ a=h=1/4 $.}
\label{fig:indicatrices}
\end{figure}

\begin{rem}
\label{rem:matsumoto}
For each $t$, the formula  \eqref{eq:F_p_no_wind} agrees with the expression of a 
{\em Matsumoto metric} on $N$. Indeed,  Matsumoto metrics 
are of the form\footnote{Matsumoto metrics may also refer just to the normalized ones, type
$
F(v) = \frac{\alpha(v)^2}{\alpha(v)-\beta(v)}.
$}
$$
\hat F(v) =  \frac{\alpha(v)^2}{b \; \alpha(v)-c \; \beta(v)},
$$
being $ \alpha $ the norm of a Riemannian metric, $\beta $ a one-form and $b,c$ any positive real functions on $N$. Such a metric effectively measures the walking distance when there is a slope. Specifically, if we put $ c:=g/2 $, where $ g $ is the gravity of Earth, solving $ \hat F(v)=1 $ for $ \alpha(v) $ returns the distance covered after one time unit by a person walking with initial speed $ b $ in the direction $ v $ on a plane with an angle of inclination $ \sigma $ given by \eqref{eq:sigma} below (see \cite{SS}).

In our case, if we identify $ b $ with $ a+h $, $ c $ with $ h $ and keep $ \alpha,\beta $ as in \eqref{eq:alpha_beta}, the only difference between the Finsler metric $ F $ in \eqref{eq:F_p_no_wind} and $ \hat F $ is the ``$ \pm $'' sign in front of $ \beta $. The ``$ - $'' sign in the expression of $ \hat F $ means that the walking distance is greater going downwards. On the contrary, we have found a ``$ + $'' sign in $ F $ because the fire moves faster in the upward direction.

Summing up, our model for the fire spread over a surface is analogous to that of a person walking on the same surface, but exchanging the upward and downward directions. From this point of view, $ a+h $ plays the role of the initial speed $ b $, while $ h $ (resp. $ g $) represents the reason why it is more effective to go upwards (resp. downwards).

\end{rem}

\subsubsection{Finslerian restriction}
In order to ensure that a Matsumoto metric is a Finsler one, a restriction must be imposed. Otherwise the indicatrix might be non-convex at some points. In our framework, this condition is just a bound on the slope (see \eqref{eq:cond_no_wind} below), and a detailed proof is carried out for the convenience of the reader. Such a condition will be scarcely restrictive; however, in Appendix~\ref{appendix2} we discuss the appropriate modifications may the condition not be met.
 
\begin{defi}
\label{def:slant}
The {\em slant function} $ \sigma $ is the positive real function on $ N $ such that, for each $ p = (x,y) \in N $, $ \sigma(p) $ is the angle of inclination, i.e., the angle between the tangent plane $ T_{\hat{z}(p)}\hat N $ and the horizontal plane $ z = 0 $ (see Fig.~\ref{fig:slope}). It is given by
\begin{equation}
\label{eq:sigma}
\cos(\sigma(p)) = \frac{1}{\sqrt{1+(\partial_xz|_p)^2+(\partial_yz|_p)^2}}.
\end{equation}
\end{defi}

\begin{rem}
\label{rem:delta_no_wind}
The slant function $ \sigma $ is related to $\delta$. Indeed, observe that
\begin{equation*}
\sigma(p) = \frac{\pi}{2}-\delta_{min}(p),
\end{equation*}
where $ \delta_{min}(p) := \min \lbrace \delta(p,\theta_v): v \in T_pN \rbrace = \delta(p,\theta_{\nabla z}) $, with $ \nabla z = (\partial_x z,\partial_y z) $, i.e., $ \delta $ reaches its mininum in the direction of maximum slope $ \df \hat{z}(\nabla z) $; see Fig.~\ref{fig:slope}. Conversely, we obtain the maximum $ \delta $ in the opposite direction $ \df \hat{z}(-\nabla z) $, namely, $ \delta_{max}(p)=\delta(p,\theta_{-\nabla z})=\pi-\delta_{min}(p) $.
\end{rem}
 
\begin{prop}
\label{prop:F_no_wind}
$ F: \textup{Ker}(dt) \rightarrow [0,\infty) $, defined pointwise by \eqref{eq:F_p_no_wind} with $ F(0) := 0 $, is a Finsler metric at each $t\in\R$ if and only if the slant function satisfies
\begin{equation}
\label{eq:cond_no_wind}
2\sin(\sigma(p)) < \frac{a(t,p)+h(t,p)}{h(t,p)}, \quad \forall t \in \R, p \in N.
\end{equation}
As a sufficient condition, this holds if $ a(t,p) > h(t,p) $ for all $ t\in \R, p \in N $.
\end{prop}
\begin{proof}
Since we are assuming that the functions $ \hat{z}(p) $, $ a(t,p) $ and $ h(t,p) $ vary smoothly with $ t $ and $ p $, it suffices to check the conditions under which $ F_{(t,p)} $ is a Minkowski norm for all $ t \in \R $, $ p \in N $ (recall Def.~\ref{def:finsler}(i)):
\begin{itemize}
\item $ F_{(t,p)} $ is always positive, since we have constructed $ s_{fire} $ in \eqref{eq:mod_no_wind} to be so. To see this explicitly, note first that
\begin{equation*}
\cos(\delta(p,\theta_v)) = \frac{\langle \df\hat z_p(v),\partial_z|_p \rangle}{||\df\hat z_p(v)||} = \frac{\beta_p(v)}{\alpha_p(v)}
\end{equation*}
and thus, from \eqref{eq:F_p_no_wind}, $ F_{(t,p)} $ is positive on $ \textup{Ker}(dt_{(t,p)})\setminus\{0\} $ if and only if
\begin{equation*}
\cos(\delta(p,\theta_v)) > -\frac{a(t,p)+h(t,p)}{h(t,p)}, \quad \forall v \in \textup{Ker}(dt_{(t,p)})\setminus\{0\},
\end{equation*}
which always holds. Moreover, since $ \alpha_p $ and $ \beta_p $ are both smooth, so is $ F_{(t,p)} $ on $ \textup{Ker}(dt_{(t,p)})\setminus\{0\} $. Also, the definition $ F_{(t,p)}(0) := 0 $ makes $ F $ continuous and $ F_{(t,p)}(v) = 0 \Leftrightarrow v = 0 $.

\item $ F_{(t,p)} $ is always positive homogeneous, since $ \alpha_p(\lambda v) = \lambda\alpha_p(v) $ and $ \beta_p(\lambda v) = \lambda\beta_p(v) $ for all $ \lambda > 0 $. This also guarantees that $ \Sigma_{(t,p)} $ is a closed smooth hypersurface embedded in $ \textup{Ker}(dt_{(t,p)}) $, as $ 1 $ is a regular value of $ F_{(t,p)} $.

\item From \cite[Cor. 4.15]{JS14} (applied to $ (a+h)^2F $), the fundamental tensor of $ F_{(t,p)} $ is positive definite in every direction or, equivalently, the indicatrix $ \Sigma_{(t,p)} $ is strongly convex if and only if
\begin{equation*}
(a(t,p)+h(t,p))\alpha_p(v)+2h(t,p)\beta_p(v) > 0, \quad \forall v \in \textup{Ker}(dt_{(t,p)})\setminus\{0\},
\end{equation*}
which is equivalent to
\begin{equation*}
2\cos(\delta(p,\theta_v)) > -\frac{a(t,p)+h(t,p)}{h(t,p)}, \quad \forall v \in \textup{Ker}(dt_{(t,p)})\setminus\{0\}.
\end{equation*}
Observe that this condition is more restrictive the bigger $ \delta $ is, with $ \delta \in (0,\pi) $. In fact, from Rem.~\ref{rem:delta_no_wind}, $ \delta_{max}=\pi-\delta_{min}=\frac{\pi}{2}+\sigma $ and using that $ \cos(\frac{\pi}{2}+\sigma)=-\sin\sigma $, the previous condition is equivalent to \eqref{eq:cond_no_wind}.
\end{itemize}
\end{proof}

\subsection{Including the effect of the wind}
\label{subsec:wind}
The effect of the wind might be quite subtle. If one considers, for example, the propagation of a sound wave in the air, a reasonable approximation might be just to displace the indicatrix $\Sigma$, translating it with the vector field that represents the velocity of the wind (see \cite{GW10,GW11}). In the case of a wildfire, however, the situation is much more complex (see, e.g., \cite{GHDM}) and empirical models are needed. Among the contributions, we will focus on Anderson's double semi-elliptical model without slope \cite{An} and the natural choice of considering the ignition point as one of the focal points in the elliptical case \cite{Al,F}. This last assumption, although reasonable, might not be necessarily true in real wildfires; anyway, it is certainly the simplest option as a first approximation, with other choices over-complicating the model. Having this in mind, our proposed fire speed is now
\begin{equation}
\label{eq:mod_wind}
s_{fire}(t,p,\theta) = \frac{a(t,p)[1-\varepsilon(t,p)^2]}{1-\varepsilon(t,p)\cos(\tilde\theta(p,\theta)-\tilde\phi(t,p))} + h(t,p)[1+\cos(\delta(p,\theta))],
\end{equation}
where we have maintained the same contribution of the slope, but the first term has been modified to be elliptical instead of spherical (compare with \eqref{eq:mod_no_wind}). Specifically:
\begin{itemize}
\item Let $ E \subset \R^2 $ be an ellipse of eccentricity $ \varepsilon \in [0,1) $ centered at one of its foci, with its semi-major axis $ a $ oriented in the direction $ \phi \in [0,2\pi) $ (angle with respect to the $ x $-axis). Then $ \frac{a(1-\varepsilon^2)}{1-\varepsilon\cos(\theta-\phi)} $ is the length of the vector from the origin (focus) to $ E $ in the direction $ \theta $. Specifically, the centered focus must be the one such that when $ \phi=0 $, the direction $ \theta=0 $ points towards the other focus (see Fig.~\ref{fig:wind_a}). 

\item When the ellipse is on an inclined plane, $ E \subset T_{\hat z(p)}\hat N \subset \R^3 $ (the case when there is a slope), then the angles $ \phi,\theta $, which measure the aerial direction on $ T_pN \subset \R^2 $ with respect to the $ x $-axis, must be exchanged by $ \tilde\phi,\tilde\theta $, which measure the actual direction on $ T_{\hat z(p)}\hat N $ with respect to $ \df \hat z(\partial_x) $. Using \eqref{eq:u}, note that the aerial and actual directions are related by
\begin{equation*}
\cos\tilde\phi = \frac{\langle u_{\phi}|_p,\df \hat z_p(\partial_x) \rangle}{||\df \hat z_p(\partial_x)||} = \frac{\cos\phi+\cos\phi(\partial_xz|_p)^2+\sin\phi\partial_xz|_p\partial_yz|_p}{\sqrt{1+(\cos\phi\partial_xz|_p+\sin\phi\partial_yz|_p)^2}\sqrt{1+(\partial_xz|_p)^2}},
\end{equation*}
\begin{equation*}
\sin\tilde\phi = \frac{\langle u_{\phi}|_p,\df \hat z_p(\partial_y) \rangle}{||\df \hat z_p(\partial_y)||} = \frac{\sin\phi+\sin\phi(\partial_yz|_p)^2+\cos\phi\partial_xz|_p\partial_yz|_p}{\sqrt{1+(\cos\phi\partial_xz|_p+\sin\phi\partial_yz|_p)^2}\sqrt{1+(\partial_yz|_p)^2}}
\end{equation*}
(the same equations also apply to $ \tilde\theta $ and $ \theta $). Then the first term on the right-hand side of \eqref{eq:mod_wind} is the length of the vector from the origin to $ E $ in the aerial direction $ \theta $.

\item When the elliptical term is added to the term $ h $ in \eqref{eq:mod_wind} (neglecting $ \cos\delta $), then the shape obtained becomes the directional sum of the ellipse $ E $ and the sphere $ S_h $ of radius $ h $ centered at the origin.\footnote{By ``directional sum'' of $ E $ and $ S_h $ we mean that for each (oriented) direction $ \theta $, we sum the vectors that go from the origin to $ E $ and $ S_h $ in that direction $ \theta $.} This resembles a double semi-ellipse (see $ E+S_h $ in Fig.~\ref{fig:wind_b}), argued as the most realistic representation without slope (observe the fire growth fittings of real wildfires in \cite{An}).

\item Finally, the last term $ h\cos\delta $ implements the contribution of the slope in the same way as above (see $ \Sigma $ in Fig.~\ref{fig:wind_b}).
\end{itemize}

\begin{figure}
\centering
\subfigure[Elliptical term of the model without slope.]{\label{fig:wind_a}\includegraphics[width=0.49\textwidth]{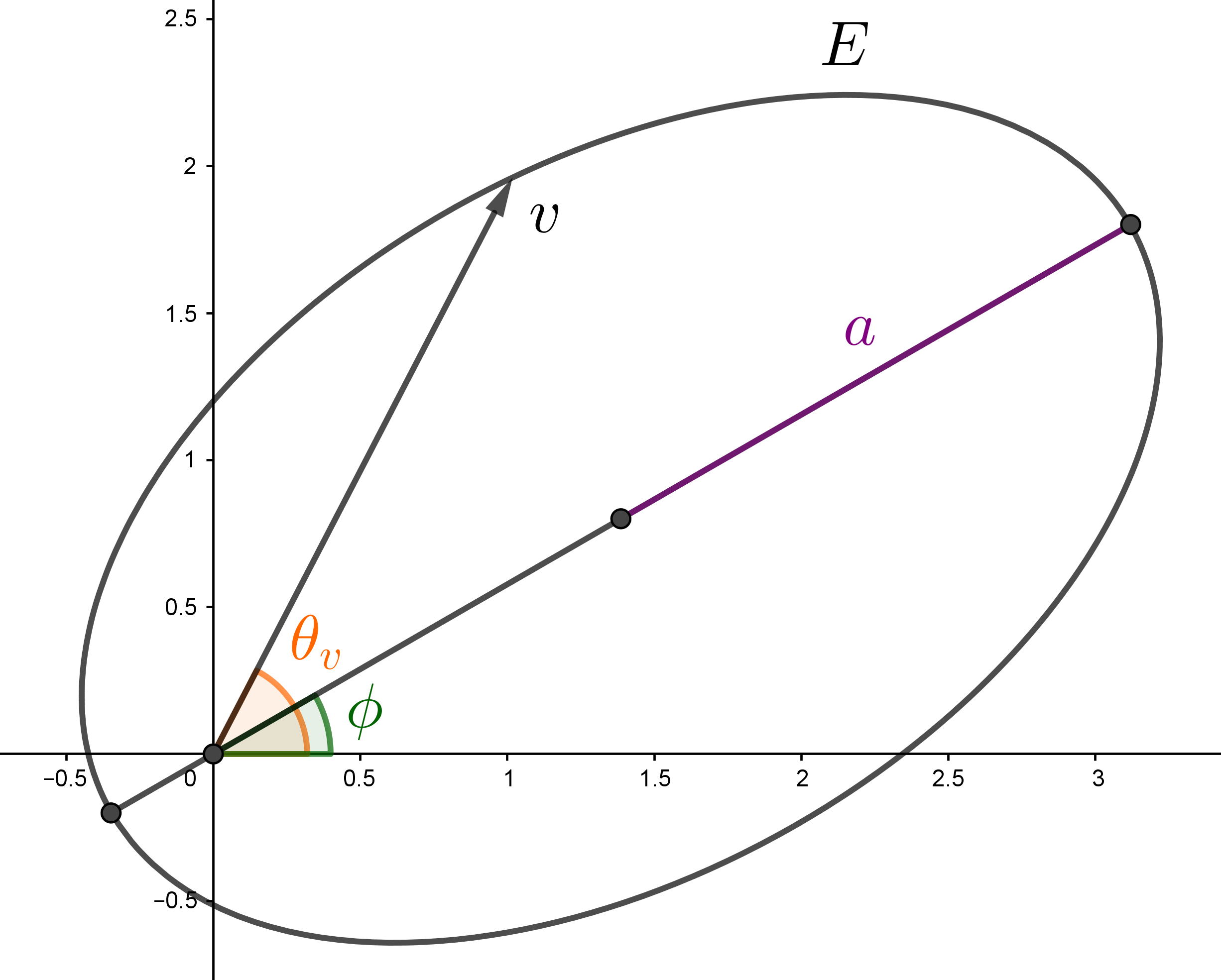}}
\subfigure[Comparison between the cases with and without slope.]{\label{fig:wind_b}\includegraphics[width=0.49\textwidth]{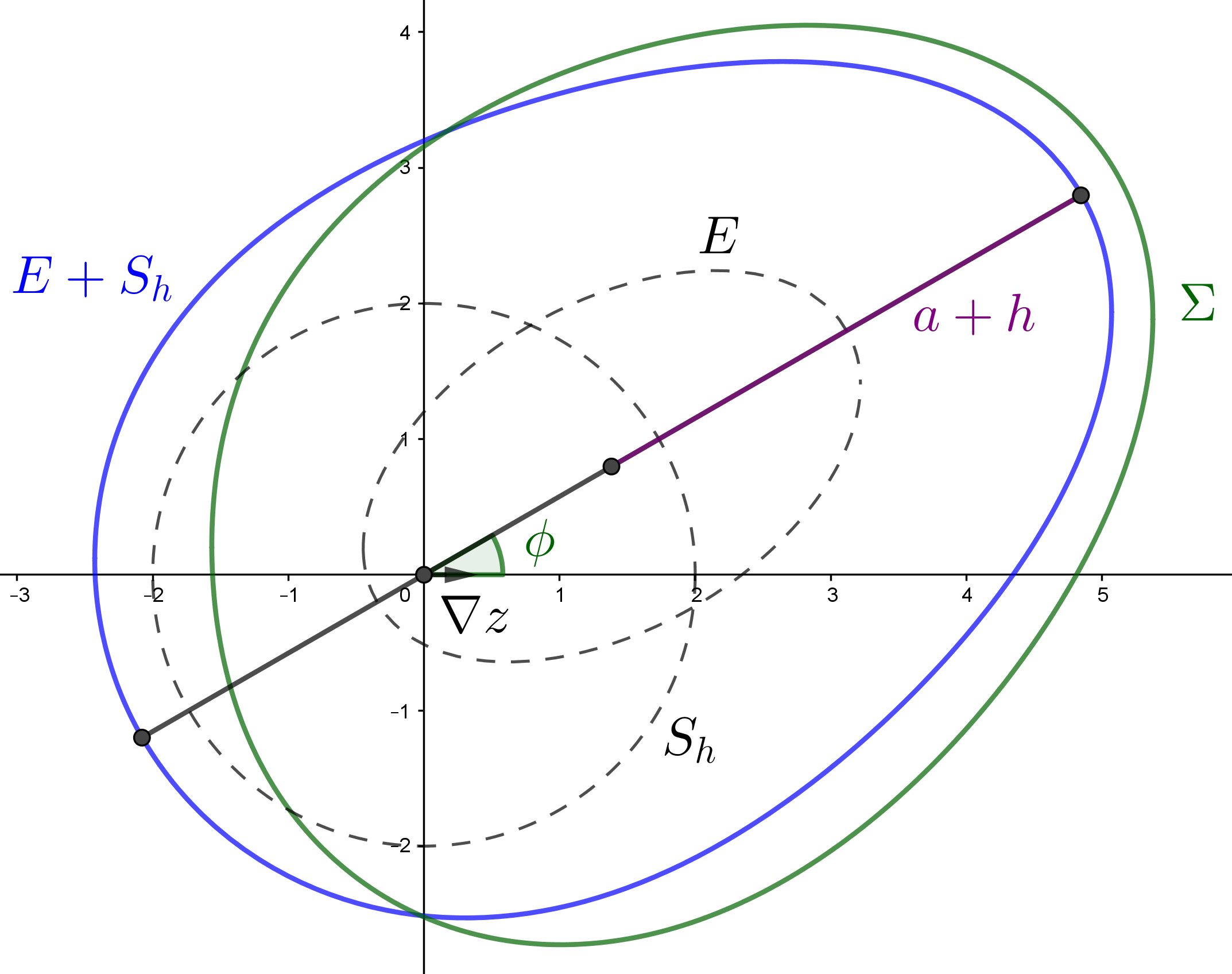}}
\caption{On the left, the underlying ellipse $ E $ corresponding to the first term in \eqref{eq:mod_wind} if there was no slope (so that $ \phi=\tilde\phi $ and $ \theta=\tilde\theta $). It is also depicted on the right in order to appreciate the shift generated by the other contributions. When the sphere $ S_h $ is added, one obtains a shape similar to a double semi-ellipse with semi-major axis $ a+h $, which would be the indicatrix without slope. Finally, $ \Sigma $ incorporates the effect of the slope, with $ \nabla z $ indicating the direction of maximum slope. Note that $ E+S_h $ provides the actual velocities (in the case without slope), while $ \Sigma $ is a projection and contains the aerial ones. Data: $ \partial_xz=2/5 $, $ \partial_yz=0 $, $ a=h=2 $, $ \varepsilon=0.8 $ and $ \phi=\pi/6 $.}
\label{fig:wind}
\end{figure}

\begin{rem}
Recall that in \eqref{eq:mod_wind} the wind and slope terms directly represent the speed of the fire induced by both phenomena. This makes it possible for the two terms to be added together, despite the different nature of both effects.\footnote{Another natural option would be to consider the Minkowski sum of both contributions, although the explicit expression of the Finsler metric in this case would become much more complicated.} In fact, the wind does not appear explicitly in our model and we do not specify how every function depends on the fuel and meteorological conditions, nor do we say how to construct the double semi-ellipse from them. Obviously, this issue has to be approached from a more physical and empirical viewpoint, and therefore it goes beyond the scope of this work. In fact, the phenomena that take place in a wildfire are quite complex. Physically, the heat transfer is carried out by three fundamental mechanisms: convention, radiation and conduction, and behind each one there is a complete physical model (see, e.g., \cite{Rot}, which is the main physical model used by FARSITE to construct the ellipse at each point).

Anyway, our approach here is general in the sense that all the unknown functions can be adapted to fit empirical data and roughly account for all the physical mechanisms. The only features intrinsic to our model are: (a) without slope, the shape of the indicatrix (infinitesimal propagation of the fire) at each point can be represented by the (directional) sum of an ellipse and a sphere, with the ignition point being one of the foci of the ellipse, (b) the spherical term couples with the contribution of the slope,\footnote{This is because $ h $ appears both in the double semi-ellipse and the slope term. This coupling ensures that $ s_{fire} $ in \eqref{eq:mod_wind} is always positive. It is not restrictive though, as we can use different functions for both terms, i.e., replace $ h(1+\cos\delta) $ in \eqref{eq:mod_wind} with $ h+h'\cos\delta $, being $ h'(t,p) $ a positive function independent of $ h(t,p) $. From a practical viewpoint, this may become relevant when trying to fit empirical data.} and (c) this contribution comes from a Matsumoto-type metric.

Finally, it is worth mentioning that the usual elliptical approximation (e.g., the one used by FARSITE) can be recovered in our model by setting $ h \equiv 0 $, with the obvious loss of the Matsumoto term to model the slope.
\end{rem}

Using \eqref{eq:delta}, \eqref{eq:u} and the trigonometric identity $ \cos(\tilde\theta-\tilde\phi)=\cos\tilde\theta\cos\tilde\phi+\sin\tilde\theta\sin\tilde\phi $, \eqref{eq:mod_wind} is equivalent to
\begin{equation*}
s_{fire}(t,p,\theta) = \frac{a(1-\varepsilon^2)}{1-\varepsilon(\cos\tilde\theta\cos\tilde\phi+\sin\tilde\theta\sin\tilde\phi)} + h(1+\langle u_{\theta},\partial_z \rangle),
\end{equation*}
and therefore, a vector $ v=(v_1,v_2) \in T_pN $ represents the aerial velocity of the fire, i.e., $ ||\df \hat z_p(v)||=s_{fire}(t,p,\theta_v) $, if and only if
\begin{equation}
\label{eq:alpha_wind}
\alpha(v)^2=\alpha(v)\frac{a(1-\varepsilon^2)}{1-\varepsilon(\cos\tilde\theta_v\cos\tilde\phi+\sin\tilde\theta_v\sin\tilde\phi)}+h(\alpha(v)+\beta(v)),
\end{equation}
where $ \alpha $ and $ \beta $ are given by \eqref{eq:alpha_beta}, and $ \tilde{\theta}_v $ is the angle between $ \df \hat z_p(v) $ and $ \df \hat z_p(\partial_x) $, namely,
\begin{equation*}
\cos\tilde\theta_v = \frac{\langle \df \hat z_p(v),\df \hat z_p(\partial_x) \rangle}{\alpha_p(v)||\df \hat z_p(\partial_x)||} = \frac{v_1+\partial_xz|_p\beta_p(v)}{\alpha_p(v)\sqrt{1+(\partial_xz|_p)^2}},
\end{equation*}
\begin{equation*}
\sin\tilde\theta_v = \frac{\langle \df \hat z_p(v),\df \hat z_p(\partial_y) \rangle}{\alpha_p(v)||\df \hat z_p(\partial_y)||} = \frac{v_2+\partial_yz|_p\beta_p(v)}{\alpha_p(v)\sqrt{1+(\partial_yz|_p)^2}}.
\end{equation*}
So, if we define the one-form
\begin{equation*}
\omega_{(t,p)}(v) := \frac{v_1+\partial_xz|_p\beta_p(v)}{\sqrt{1+(\partial_xz|_p)^2}}\cos\tilde\phi(t,p)+\frac{v_2+\partial_yz|_p\beta_p(v)}{\sqrt{1+(\partial_yz|_p)^2}}\sin\tilde\phi(t,p),
\end{equation*}
then \eqref{eq:alpha_wind} reduces to
\begin{equation*}
\alpha(v)^2 = \alpha(v)^2\frac{a(1-\varepsilon^2)}{\alpha(v)-\varepsilon\omega(v)}+h(\alpha(v)+\beta(v)).
\end{equation*}
Following the same steps as above, we arrive at a similar expression for the indicatrix $ \Sigma_{(t,p)} $, which is given by \eqref{eq:indx_matsumoto} with
\begin{equation*}
Q_{(t,p)}(v) := \alpha_p(v)^2\left(1-\frac{a(t,p)(1-\varepsilon(t,p)^2)}{\alpha_p(v)-\varepsilon(t,p)\omega_{(t,p)}(v)}\right)-h(t,p)(\alpha_p(v)+\beta_p(v)),
\end{equation*}
and for the metric $ F $:
\begin{equation}
\label{eq:F_p_wind}
F_{(t,p)}(v) = \frac{\alpha_p(v)^2}{\alpha_p(v)^2\frac{a(t,p)(1-\varepsilon(t,p)^2)}{\alpha_p(v)-\varepsilon(t,p)\omega_{(t,p)}(v)}+h(t,p)(\alpha_p(v)+\beta_p(v))},
\end{equation}
for all $ v \in \textup{Ker}(dt_{(t,p)})\setminus\{0\} $.

\begin{rem}
\label{rem:convex}
Note that \eqref{eq:F_p_wind} reduces to the Matsumoto metric \eqref{eq:F_p_no_wind} when $ \varepsilon=0 $ (i.e., without wind). This ensures that, as long as \eqref{eq:cond_no_wind} holds, \eqref{eq:F_p_wind} is a Finsler metric for small $ \varepsilon $. The general condition to guarantee the strong convexity of the metric, however, is not as straightforward as in the situation without wind and greatly depends on $ \varepsilon $ and the orientation $ \tilde\phi $. Examples can be easily found where \eqref{eq:cond_no_wind} holds but \eqref{eq:F_p_wind} fails to be strongly convex, and vice versa, cases where \eqref{eq:F_p_no_wind} is not strongly convex, but certain values of $ \varepsilon $ and $ \tilde\phi $ make \eqref{eq:F_p_wind} be a proper Finsler metric (particularly if the wind blows opposite to the direction of maximum slope).
\end{rem}


\section{Computation of the firefront}
\label{sec:calculation}
Given the Finsler metric $ F $ whose indicatrix $ \Sigma $ contains the velocity of the fire at each point and time, we can compute the wildfire propagation following the general procedure for waves developed in \cite{JPS}. This characterizes the fastest trajectories of the wave propagation as geodesics of a suitable metric (equation \eqref{eq:ode} below). Here we will summarize the approach applied to our particular case when $ F $ is given by \eqref{eq:F_p_wind}.

\subsection{Spacetime viewpoint} \label{ss_spacetime_viewpoint}
Generalized Huygens' envelope principle characterizes the wildfire propagation by means of a hyperbolic PDE (e.g., Richards' equations \cite{R} when the propagation is elliptical). However, using  $ \Sigma $ and the time coordinate $t$, we can construct a pointwise cone of maximum speed of propagation for any wave in the spacetime framework. This resembles the relativistic picture of nullcones to model the propagation of light in General Relativity, whose methods are applied now.

If the fire starts from a region $B_0\subset \{t=0\}$ delimited by a closed curve $S_0$, then it will propagate to some region $J^+(B_0) \subset  M=\R\times N$,  the (relativistic) {\em causal future} of $B_0$. Moreover, its boundary $ \partial J^+(B_0)$ yields the outermost trajectories of the fire at each $ t_0 > 0 $, namely, $ \partial J^+(B_0)\cap\{t=t_0\} $ provides the firefront at $ t_0 $ (see \cite[\S~3]{JPS}). The {\em firemap}\footnote{This is called in general {\em wavemap} in \cite{JPS}. The notational distinction between $S$ and $S_0$ therein is  irrelevant here.}
$$
\hat{f}: [0,+\infty)\times S_0 \rightarrow M, \qquad (t,s) \mapsto \hat f(t,s)=(t,f(t,s)),
$$
is defined such that, for each $ s \in S_0 $, $ t \mapsto \hat f(t,s) $ is the $ t $-parametrized cone geodesic of the Finsler spacetime\footnote{See \cite[\S~3]{JS20} and \cite[\S~2]{JPS} for background regarding Finsler spacetimes and this type of Lorentz-Finsler metrics. In particular, we assume $ (M,G) $ to be globally hyperbolic, with every slice $ \{t=t_0\} $ being a Cauchy hypersurface (i.e., a surface which is crossed exactly once by any inextendible causal curve); see \cite[Rem. 3.2]{JPS}.} $ (M,G:=dt^2-F^2) $ with initial velocity the unique lightlike vector $ G $-orthogonal to $ S_0 $ and pointing out to the exterior of $ S_0 $.\footnote{This is equivalent to saying that the projection of the initial velocity on $ TN $ is $ F $-orthogonal to $ S_0 $ and points out to the exterior of $ S_0 $.} Then, the curve $ t \mapsto \hat f(t,s_0)=(t,f(t,s_0)) $ represents the {\em spacetime trajectory} of the fire from $ s_0\in S_0 $, being its projection $ t \mapsto f(t,s_0) $ the {\em spatial trajectory} on $ N $. Each of these curves remains entirely in $ \partial J^+(B_0) $ and minimizes the propagation time in the domain of its {\em (null) cut function}
\begin{equation}
\label{eq:cut}
c: S_0\rightarrow [0,\infty], \qquad s\mapsto c(s):=\text{Max}\{t:\hat f(t,s)\in \partial J^+(B_0)\}.
\end{equation}
If $ c(s_0)=t_0 < \infty $ for some $ s_0 \in S_0 $, then $ t_0 $ and $ \hat f(t_0,s_0) $ are, respectively, the {\em cut instant} and {\em cut point} of the corresponding wildfire trajectory.

It can be shown that there always exists a time $ \epsilon>0 $ below which no cut points appear (see \cite[Thm. 4.8]{JPS}). This ensures that at least in a short lapse $ [0,\epsilon] $, the firemap maps $ \partial J^+(B_0) $ (and this portion of $ \partial J^+(B_0) $ becomes a smooth embedded hypersurface of $ M $ up to $ B_0 $), i.e., for any fixed $ t_0\in[0,\epsilon] $, the curve $ s \mapsto \hat f(t_0,s) $ is the firefront at $ t_0 $. So, we can directly determine how the fire evolves by computing the firemap $ \hat f $, at least until $ t=\epsilon $. Moreover, the role of the cut function will be discussed in \S~\ref{subsec:cut_points} (see also Appendix~\ref{appendix1}) and we will see that the computation of the firefront can be easily extended to $ t>\epsilon $.

\subsection{PDE vs ODE systems} \label{ss_ode}
From now on, the following notation will be used, intended to simplify the expressions we will find later:
\begin{itemize}
\item $ \partial_t\hat f(t,s)=(1,\partial_tf(t,s)) $ and $ \partial_s\hat f(t,s)=(0,\partial_sf(t,s)) $ will denote the corresponding velocities (tangent vectors) of the curves $ t \mapsto \hat f(t,s) $ and $ s \mapsto \hat f(t,s) $, respectively.

\item $ (x^0,x^1,x^2):=(t,x,y) $, where $ (t,x,y) $ are the natural coordinate functions on $ M=\R\times N$.

\item $ g_{ij}(\hat v) := g^G_{\hat v}(\frac{\partial}{\partial x^i},\frac{\partial}{\partial x^j}) $ for any $ \hat v \in TM \setminus \textup{Span}(\frac{\partial}{\partial t}) $, where $ g^G $ is the fundamental tensor of $ G = dt^2-F^2 $.\footnote{Since $ G $ is a Lorentz-Finsler metric, $ g^G $ is obtained as in \eqref{eq:fund_tensor} replacing $ F^2 $ with $ G $ (see \cite[Def. 2.6]{JPS}).} In particular, we will choose $ \hat v = \partial_t\hat f $, in which case the relationship between the coefficients of $ g_{\partial_t\hat f}^G $ and those of $ g_{\partial_t f}^F $ is
\begin{equation*}
\lbrace g_{ij}(\partial_t\hat f) \rbrace = 
\begin{pmatrix}
1 & 0 & 0\\
0 & -g_{11}^F(\partial_tf) & -g_{12}^F(\partial_tf)\\
0 & -g_{12}^F(\partial_tf) & -g_{22}^F(\partial_tf)
\end{pmatrix}.
\end{equation*}

\item $ g^{ij}(\hat v) $ will denote the coefficients of the inverse matrix of $ \{g_{ij}(\hat v)\}$.

\item Finally, we will use Einstein's summation convention, omitting the sums from $ 0 $ to $ 2 $ when an index appears up and down, and raising and lowering indices using $ g_{ij} $ and $ g^{ij} $.
\end{itemize}

In order to obtain the firemap, we need to calculate the lightlike geodesics of $ (M,G) $, i.e., solve the ODE system given by the geodesic equations (see \cite[Cor. 4.13]{JPS}). Note that in these equations the {\em formal Christoffel symbols} of $G$ appear (see \eqref{eq:ode} below), which are defined as
$$
\gamma_{\ ij}^k(\hat v) := \frac{1}{2}g^{kr}(\hat v)\left(\frac{\partial g_{rj}}{\partial x^i}(\hat v)+\frac{\partial g_{ri}}{\partial x^j}(\hat v)-\frac{\partial g_{ij}}{\partial x^r}(\hat v)\right), \quad i,j,k = 0,1,2,
$$
for any $ \hat v=(\tau,v)=(\tau,v_1,v_2) \in TM\setminus \text{Span}(\frac{\partial}{\partial t}) $.

Alternatively, the firemap is also characterized in $ [0,\epsilon]\times S_0 $ by the PDE system given by the following {\em orthogonality conditions} (see \cite[Thm. 4.14]{JPS}):
\begin{equation}
\label{eq:ort_cond_F}
\left\lbrace{
\begin{array}{l}
F(\partial_tf(t,s))  = 1, \ \text{with} \ \partial_tf(t,s) \ \text{pointing outwards},\\
\partial_tf(t,s) \bot_F \partial_sf(t,s).
\end{array}
}\right.
\end{equation}
Observe that, if $S_0$ is counter-clockwise parametrized, for $ \partial_tf(t,s) $ to be pointing outwards we only need to ensure that
\begin{equation}
\label{eq:outwards}
\begin{vmatrix}\partial_t x(0,s) & \partial_s x(0,s) \\ \partial_t y(0,s) & \partial_s y(0,s) \end{vmatrix}>0
\end{equation}
at some $s\in S_0$.

Both the geodesic equations and the orthogonality conditions can be stated in general, for any wave and any propagation pattern (see \cite{JPS}). In fact, Richards' equations \cite{R} are precisely the orthogonality conditions when an elliptical growth is assumed.
The orthogonality conditions have also been found and applied previously in \cite{M16,M17}. Now we particularize both equation systems to our model, i.e., when the propagation is given by the Finsler metric \eqref{eq:F_p_wind}.

\begin{thm}
\label{th:pde_ode}
The firemap $ \hat{f}(t,s) = (t,f(t,s)) $ for a wildfire given by the Finsler metric \eqref{eq:F_p_wind} is determined in $ t \in [0,\epsilon] $ by the following PDE system for the partial derivatives $ \partial_tf(t,s) $ and $ \partial_sf(t,s) $, along with the initial condition $ f(0,s) = s \in S_0 $:
\begin{equation}
\label{eq:pde}
\left\lbrace{
\begin{split}
0 = \ & \alpha(\partial_tf)^2 \left(1-\frac{a(1-\varepsilon^2)}{\alpha(\partial_tf)-\varepsilon \ \omega(\partial_tf)}\right) - h(\alpha(\partial_tf)+\beta(\partial_tf)),\\
0 = \ & \frac{a(1-\varepsilon^2)}{[\alpha(\partial_tf)-\varepsilon \ \omega(\partial_tf)]^2} [ \varepsilon (2\langle \df\hat z(\partial_tf),\df\hat z(\partial_sf) \rangle \omega(\partial_tf) \\
& -\alpha(\partial_tf)^2\omega(\partial_sf)) - \alpha(\partial_tf)\langle \df\hat z(\partial_tf),\df\hat z(\partial_sf) \rangle ] \\
& + 2\langle \df\hat z(\partial_tf),\df\hat z(\partial_sf) \rangle - h \left( \frac{\langle \df\hat z(\partial_tf),\df\hat z(\partial_sf) \rangle}{\alpha(\partial_tf)}+\beta(\partial_sf)\right),
\end{split}
}\right.
\end{equation}
with $ \partial_tf $ pointing outwards (i.e., satisfying \eqref{eq:outwards}). Alternatively, $ \hat{f}(t,s) = (t,x^1(t,s),x^2(t,s)) $, with $ x^0=t $, is determined in $ t \in [0,\infty) $ by the following ODE system:
\begin{equation}
\label{eq:ode}
\partial_t^2x^k = -\gamma_{\ ij}^k(\partial_t\hat{f})\partial_tx^i\partial_tx^j + \gamma_{\ ij}^0(\partial_t\hat{f})\partial_tx^i\partial_tx^j\partial_tx^k, \quad k=1,2,
\end{equation}
along with the initial conditions
\begin{itemize}
\item $ f(0,s) = s \in S_0 $,
\item $ \partial_tf(0,s) $ is the unique $ F $-unit vector $ F $-orthogonal to $ \partial_sf(0,s) $ and pointing outwards, i.e., it is given by \eqref{eq:pde} and \eqref{eq:outwards} at $ t = 0 $.
\end{itemize}
\end{thm}
\begin{proof}
The first orthogonality condition $ F(\partial_tf) = 1 $ directly yields the first equation in \eqref{eq:pde}. Now, it can be easily checked from the expression of the Finsler metric \eqref{eq:F_p_wind} that its fundamental tensor $ g^F $ satisfies
\begin{equation*}
\begin{split}
g^F_v(v,u) = \ & \frac{1}{2}\frac{\partial}{\partial\delta}F(v+\delta u)^2|_{\delta=0} = F(v)\frac{\partial}{\partial\delta}F(v+\delta u)|_{\delta=0}= \\
= \ & \frac{\alpha(v)^2}{\left[ \alpha(v)^2\frac{a(1-\varepsilon^2)}{\alpha(v)-\varepsilon \ \omega(v)}+h(\alpha(v)+\beta(v)) \right]^3} \\
& \left\lbrace 2\langle \df\hat z(v),\df\hat z(u) \rangle \left[\alpha(v)^2\frac{a(1-\varepsilon^2)}{\alpha(v)-\varepsilon \ \omega(v)}+h(\alpha(v)+\beta(v))\right] \right. \\
& + \frac{a(1-\varepsilon^2)}{[\alpha(v)-\varepsilon \ \omega(v)]^2}[ \varepsilon \ \alpha(v)^2(2\langle \df\hat z(v),\df\hat z(u) \rangle \omega(v)-\alpha(v)^2\omega(u)) \\
& -\alpha(v)^3\langle \df\hat z(v),\df\hat z(u) \rangle ] \left. - h \ \alpha(v)^2\left( \frac{\langle \df\hat z(v),\df\hat z(u) \rangle}{\alpha(v)}+\beta(u)\right) \right\rbrace.
\end{split}
\end{equation*}
Therefore, the second orthogonality condition $ \partial_tf\bot_F\partial_sf $ becomes equivalent to
\begin{equation*}
\begin{split}
0 = \ & \frac{a(1-\varepsilon^2)}{[\alpha(\partial_tf)-\varepsilon \ \omega(\partial_tf)]^2} [ \varepsilon (2\langle \df\hat z(\partial_tf),\df\hat z(\partial_sf) \rangle \omega(\partial_tf) \\
& -\alpha(\partial_tf)^2\omega(\partial_sf)) - \alpha(\partial_tf)\langle \df\hat z(\partial_tf),\df\hat z(\partial_sf) \rangle ] + 2\frac{\langle \df\hat z(\partial_tf),\df\hat z(\partial_sf) \rangle}{\alpha(\partial_tf)^2} \\
& \left[\alpha(\partial_tf)^2\frac{a(1-\varepsilon^2)}{\alpha(\partial_tf)-\varepsilon \ \omega(\partial_tf)}+h(\alpha(\partial_tf)+\beta(\partial_tf)) \right] \\
& - h \left( \frac{\langle \df\hat z(\partial_tf),\df\hat z(\partial_sf) \rangle}{\alpha(\partial_tf)}+\beta(\partial_sf)\right),
\end{split}
\end{equation*}
which can be reduced to the second equation in \eqref{eq:pde} by taking into account that $ F(\partial_tf) = 1 $.

Equivalently, the firemap is also given by the geodesic equations for $ G $ parametrized by time, the general expression \eqref{eq:ode} derived in \cite[\S~4.2]{JPS}.
\end{proof}

\begin{rem}
Analytic solutions of the equation systems \eqref{eq:pde} and \eqref{eq:ode} can be obtained in some non-trivial situations. For example, Richards found explicit integral solutions to his equations under the very restricted and special condition that the fuel and meteorological conditions are time-only dependent, i.e., for each fixed time the Randers metric is a norm (see \cite[Eqs. (21) and (22)]{R2}). It can be checked that these solutions satisfy \eqref{eq:ode} (using the elliptical approximation with $ h \equiv 0 $). One straightforward way to prove this is to realize that, due to the spatial independence of the metric ($ F_{(t,p)}=F_t $), any constant spatial vector field $ v $ is Killing. So, if $ \rho $ is the (affine) parameter that makes $ \hat f(t(\rho),s_0) $ be a geodesic of $ G $, for any fixed $ s_0 \in S_0 $, then
\begin{equation*}
g^{G}_{\partial_\rho\hat f(t(\rho),s_0)}(\partial_\rho\hat f(t(\rho),s_0),v) = \partial_\rho t(\rho) g^{G}_{\partial_t\hat f(t,s_0)}(\partial_t\hat f(t,s_0),v) = \text{constant}.
\end{equation*}
Rewriting in terms of $ F_t $, with the choice $ v=\partial_s\hat f(0,s_0) = (0,\partial_s f(0,s_0)) $ and taking into account the initial conditions:
\begin{equation*}
g^{F_t}_{\partial_t f(t,s_0)}(\partial_t f(t,s_0),\partial_sf(0,s_0)) = 0, \quad \forall t \in [0,\infty), s_0 \in S_0.
\end{equation*}
This equation, together with $ F_t(\partial_tf)=1 $ ($ \partial_t\hat f $ is lightlike), enables one to obtain Richards solutions directly.
\end{rem}


\section{Examples}
\label{sec:examples}
Let us present now some simple examples to see the wildfire spread given by this model in certain situations and compare the contributions of its different elements. 

\subsection{Slope vs wind}
\label{subsec:slope_wind}
Typically in fire growth models, the effect of the slope is not treated as carefully as that of the wind. Indeed, the slope contribution is usually simplified and integrated within the overall wind effect on the fire spread. In fact, it is normally assumed that both phenomena (either combined or separate) generate an elliptical burn pattern. For instance, in FARSITE and Prometheus the slope acts as an additional wind vector that adds up to the original one, generating a ``virtual wind'' that accounts for both effects and is the one used to construct the ellipse that provides the propagation of the fire at each point (see \cite{F,TBWTA}).

However, due to the obvious differences in the physical nature of both phenomena, slope and wind are expected to contribute to the fire spread in qualitative distinct ways. This is the case in the model we have presented: the slope produces a (reverse) Matsumoto-type effect, while the wind generates a pattern similar to a double semi-ellipse. The main difference is that the wind focuses the fire trajectories much more in the direction it blows than the slope does in the steepest direction. As a consequence, a wind-driven wildfire tends to present a narrower pattern than a slope-driven one (compare the firefronts in Fig.~\ref{fig:wind_slope}). Nevertheless, the slope topography in a valley (i.e., when the slope is non-constant) can also focus the fire trajectories in dramatic ways, as evidenced by real fires and even by experiments conducted in laboratories (see, e.g., \cite{VRAAR}).



\begin{figure}
\centering
\subfigure[Slope-driven wildfire.]{\label{fig:slope_driven}\includegraphics[width=0.49\textwidth]{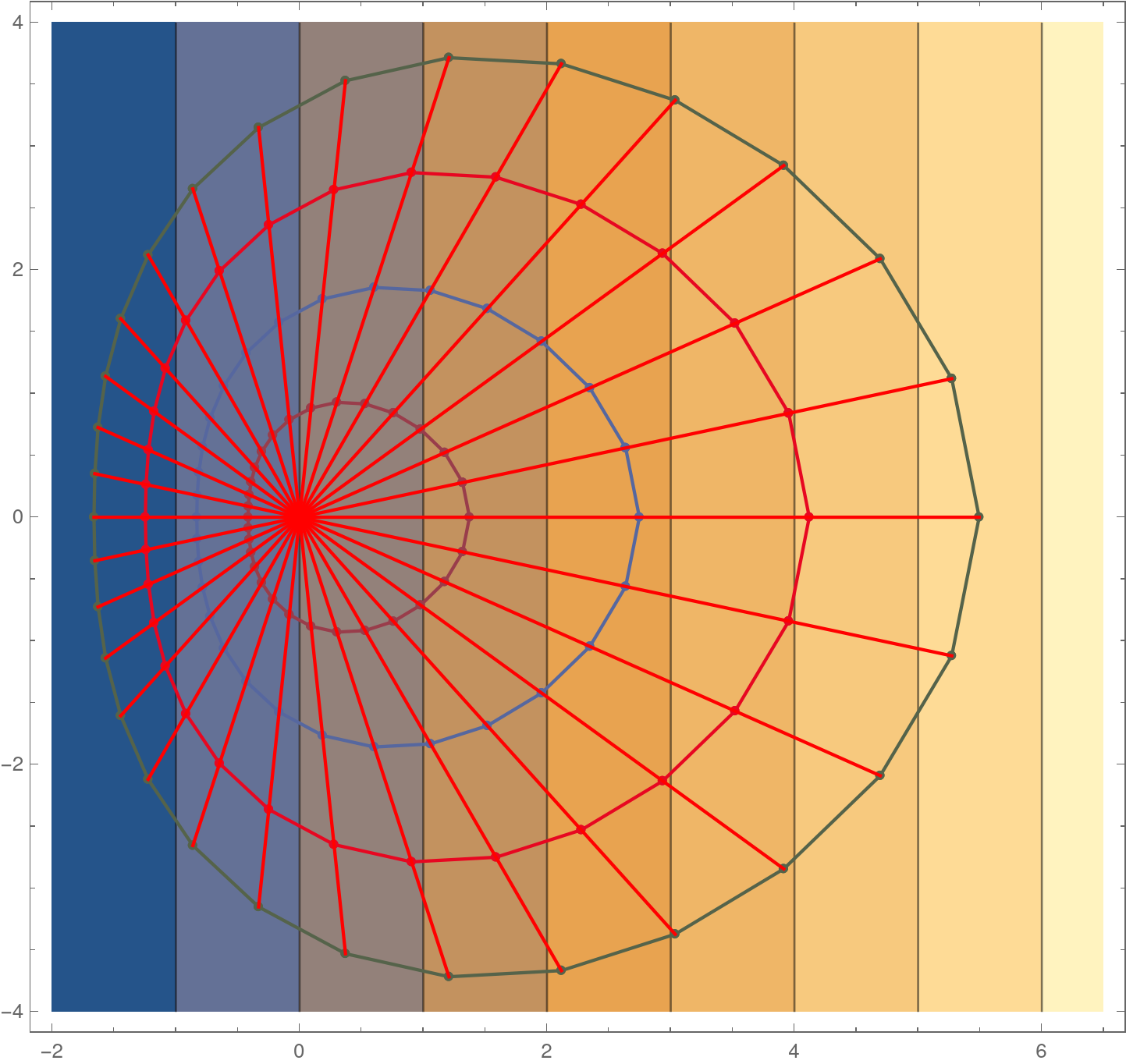}}
\subfigure[Wind-driven wildfire.]{\label{fig:wind_driven}\includegraphics[width=0.49\textwidth]{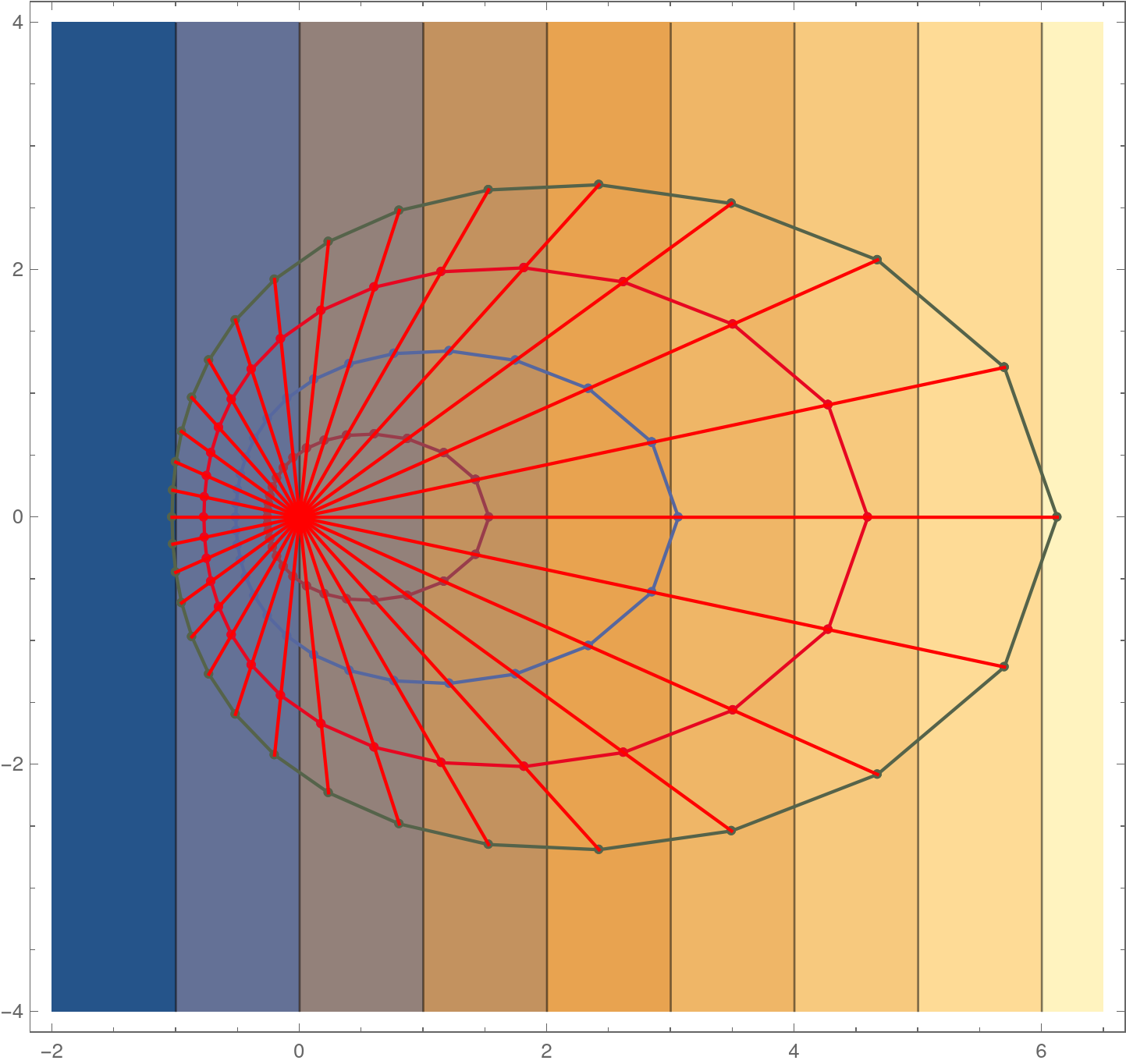}}
\caption{Comparison between the slope and wind effects on the wildfire propagation. On the left, the fire growth is governed by the contribution of the slope ($ h $ is greater than $ a $). On the right, the wind dominates ($ a $ is greater than $ h $) and the propagation becomes more focused downwind. In both examples the trajectories are straight lines because the respective metrics are constant. The firefronts join points at the same time and appear as piecewise linear; this is obviously an approximation: the more trajectories are computed, the more accurately we can reconstruct the firefront. The slope is represented by (vertical) contour lines and hypsometric tints between them.  Data: $ S_0=(0,0) $, $ z(x,y)=x/2 $, $ a=1 $ (left) or $ 3 $ (right), $ h=3 $ (left) or $ 1 $ (right), $ \varepsilon=0.8 $, $ \tilde\phi=0 $ and $ \Delta t=1 $.}
\label{fig:wind_slope}
\end{figure}

\subsection{Detection of cut points and crossovers}
\label{subsec:cut_points}
When a spacetime trajectory meets its cut point $ (t_0,p_0) \in M $, it no longer minimizes the propagation time after $ t_0 $ (or, equivalently, it no longer remains in $ \partial J^+(B_0) $) and thus, it must not be taken into account to compute the firefront beyond this time. Nevertheless, the firemap can still be calculated through the ODE system \eqref{eq:ode} at any $ t \geq 0 $ and the firefront will be given by those trajectories which have not arrived yet at their cut points. This is a major advantage of the ODE's approach, since the computation of each trajectory is independent and those reaching a cut point can be removed with no harm to the overall computation. When solving the PDE system, on the other hand, the firefront must be corrected each time there is a crossover (see Appendix~\ref{appendix1}), making the process much more expensive computationally speaking.

There are two types of cut points: the intersection of two spacetime trajectories and the focal points of the initial firefront (this is proven in Prop.~\ref{prop:appendix}). Both types are important when facing a wildfire: non-focal cut points can leave behind regions completely surrounded by the fire, while focal points mark regions where several fire trajectories come together, raising the heat intensity. So, the identification of these points becomes essential.

In practice though, when solving the ODE system \eqref{eq:ode} we can only compute a finite number of geodesics and the precise location of the cut points may become very difficult to pinpoint. However, it is not difficult to visualize the first intersection of two spatial trajectories. In this case, one of the following two situations will occur:
\begin{itemize}
\item One of the trajectories arrives strictly earlier than the other (thus, the corresponding spacetime trajectories do not meet). In this case, only the first-arriving trajectory must be taken into account for the firefront, while the other has already passed its cut point and should be discarded, see Fig.~\ref{fig:cut_points}.

\item Both spatial trajectories arrive at the same time at the intersection point, in which case it becomes the cut point of both curves (providing they have not intersected yet with any other trajectory), see Fig.~\ref{fig:cut_points_sym}.
\end{itemize}

As can be seen in the figures, the visualization of the firefront at regular time intervals allows one to find these points, which can appear in very general situations. Indeed, whenever there is a wide region where the fire speed is smaller, the firefront will go around it producing a cut point.


\begin{figure}
\centering
\subfigure[Aerial view on $ N \subset \R^2 $.]{\label{fig:cut_points_2d}\includegraphics[width=0.49\textwidth]{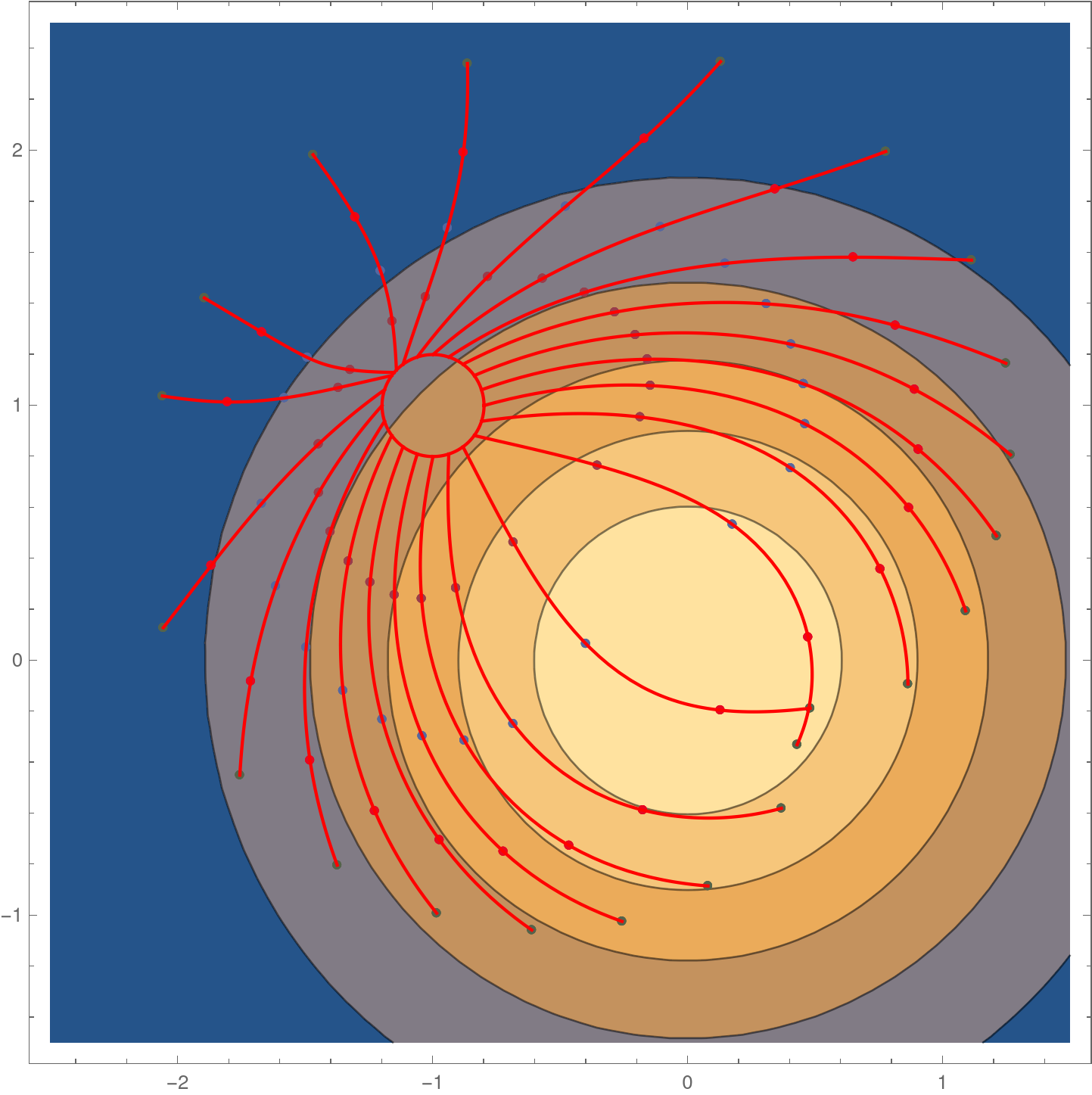}}
\subfigure[Actual 3D-propagation on $ \hat N \subset \R^3 $.]{\label{fig:cut_points_3d}\includegraphics[width=0.49\textwidth]{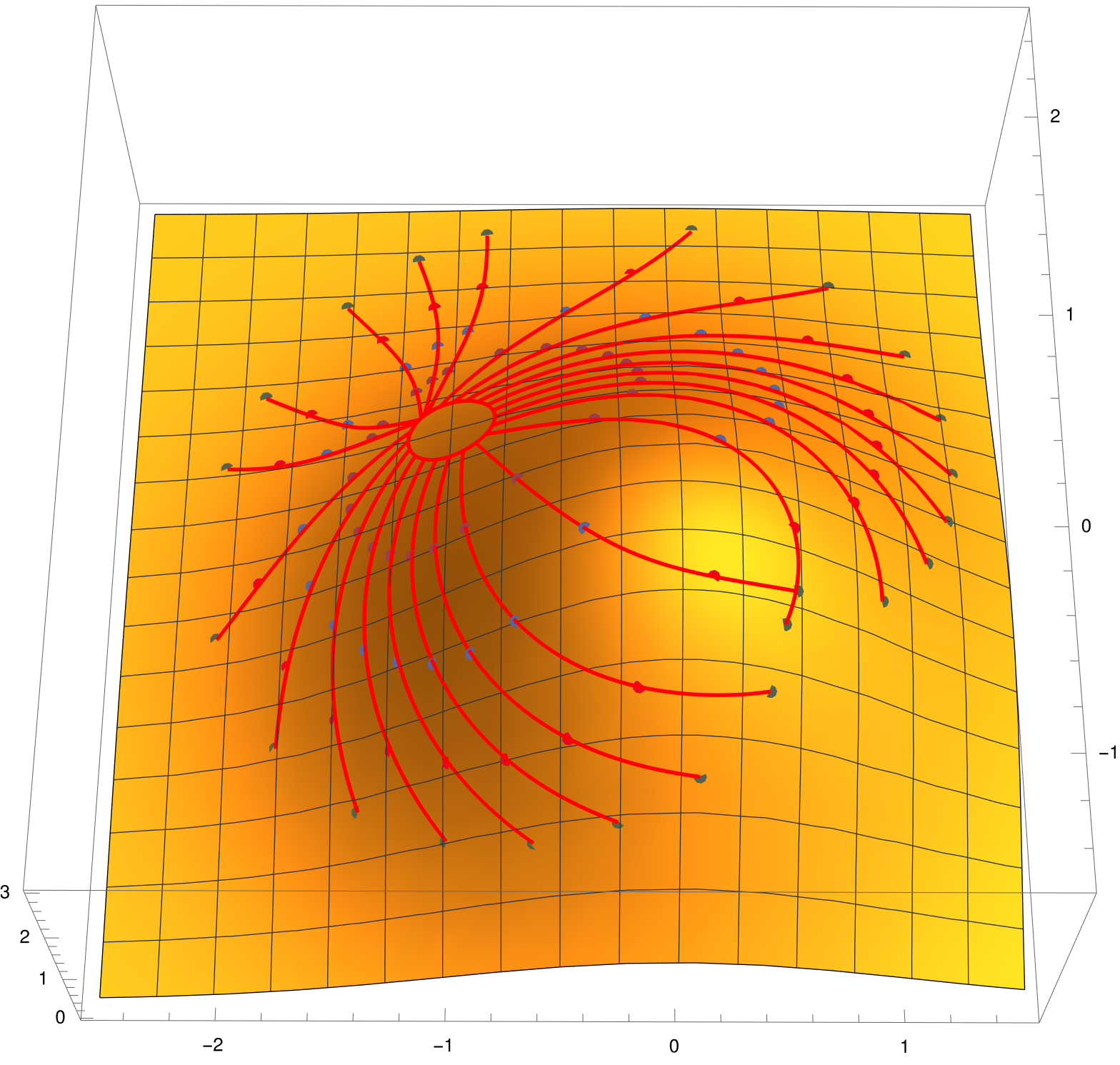}}
\caption{Situation when two spatial trajectories of the wildfire meet at different times. The one that arrives earlier (in this case, the curve that comes from above) keeps providing the firefront after the encounter. The one arriving later, however, will only pass through already ignited points, i.e., another curve from $ S_0 $ will arrive earlier at every point beyond the intersection. Data: $ S_0: [0,2\pi] \ni s\mapsto (0.2\cos s-1,0.2\sin s+1) $, $ z(x,y)=3\exp(-x^2/2-y^2/2) $, $ a=h=1 $, $ \varepsilon=0.4 $, $ \tilde\phi=0 $ and $ \Delta t=1.05 $.}
\label{fig:cut_points}
\end{figure}

\begin{figure}
\centering
\subfigure{\includegraphics[width=0.49\textwidth]{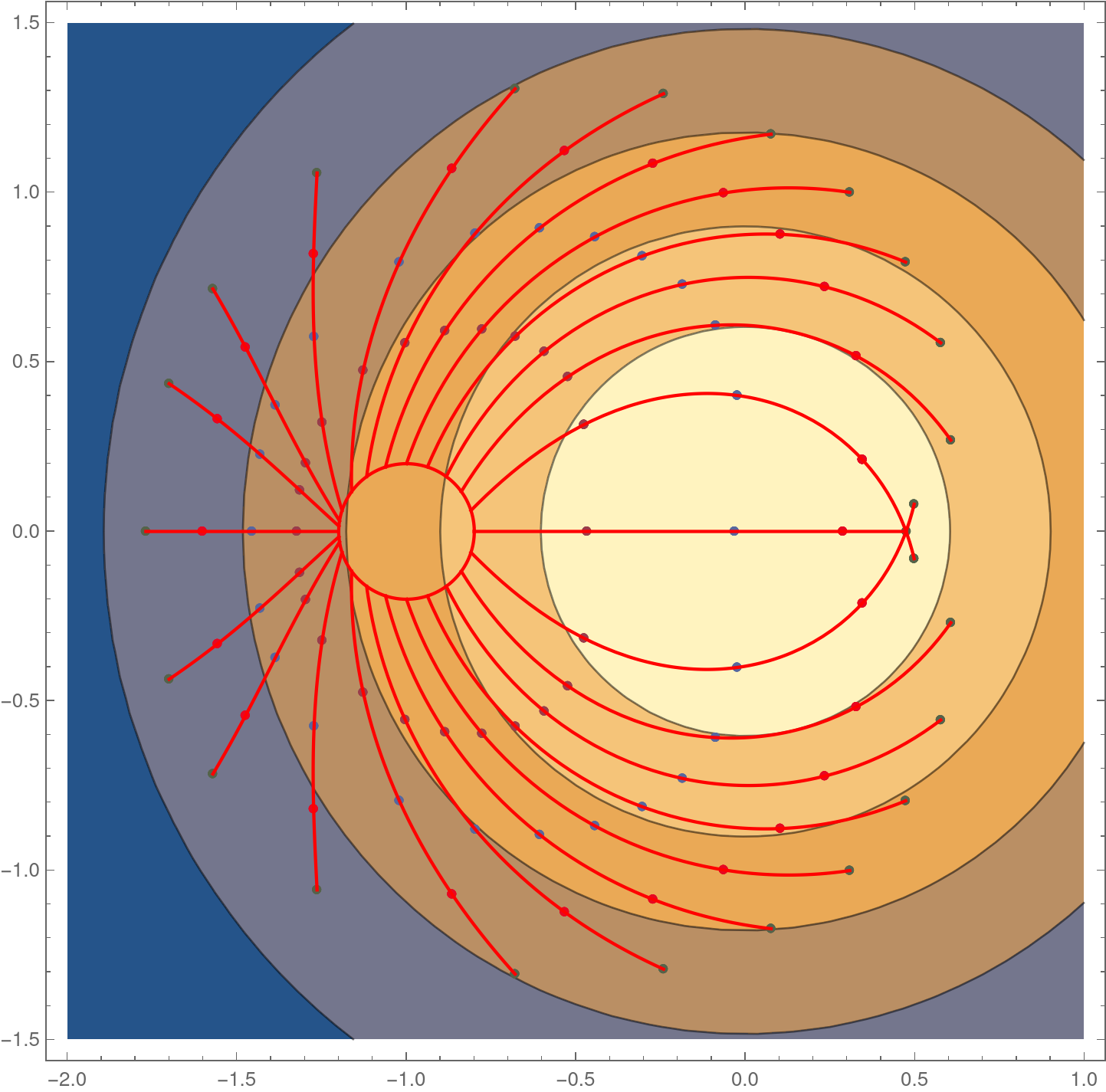}}
\subfigure{\includegraphics[width=0.49\textwidth]{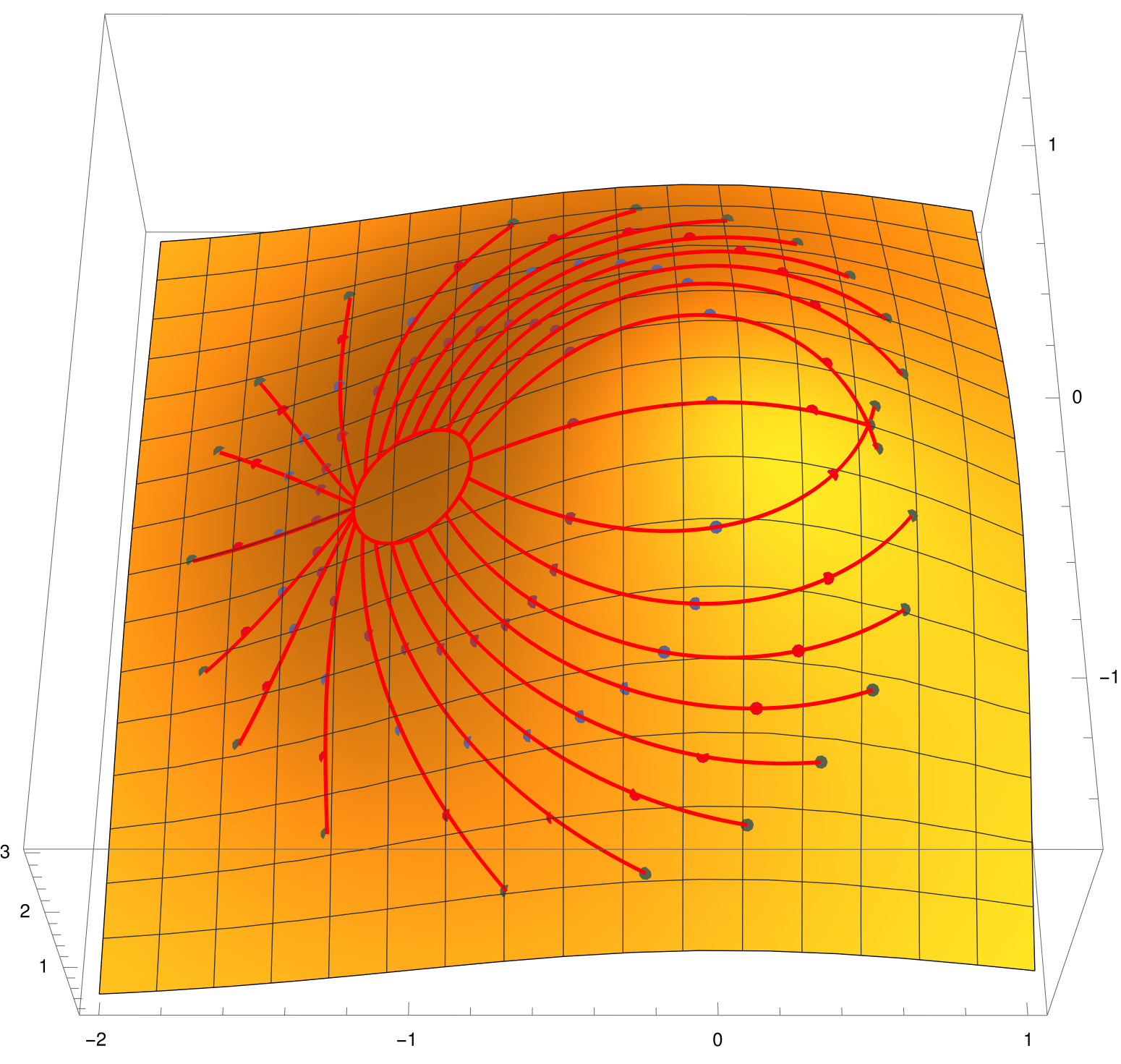}}
\caption{Wildfire with a clear symmetry in its propagation. Among the three intersecting trajectories, the two bordering the top of the mountain meet exactly at the same time, while the one going straight through the top arrives later. The intersection point becomes then the cut point of the former two curves, but not of the latter. 
Data: $ S_0: [0,2\pi] \ni s\mapsto (0.2\cos s-1,0.2\sin s) $, $ z(x,y)=3\exp(-x^2/2-y^2/2) $, $ a=h=1 $, $ \varepsilon=\tilde\phi=0 $ and $ \Delta t=0.86 $.}
\label{fig:cut_points_sym}
\end{figure}

\subsection{Flexibility: realistic wildfires}
\label{subsec:real_wildfires}
In general, realistic cases might feature a mixture of all previous effects plus more subtle ones that we have not even taken into account, therefore becoming extremely difficult to model. Although considering all the possible effects that may occur in a real wildfire is beyond the scope of this work, the model presented here is flexible enough to serve as an approximation to a great variety of situations, and also simple enough to be computable in real time. For example, Fig.~\ref{fig:complex} shows a wildfire where the wind changes over time and the fuel conditions vary from one point to another, and Fig.~\ref{fig:complex2} depicts a case with a more complicated topography.

\begin{figure}
\centering
\subfigure{\label{fig:complex_2d}\includegraphics[width=0.49\textwidth]{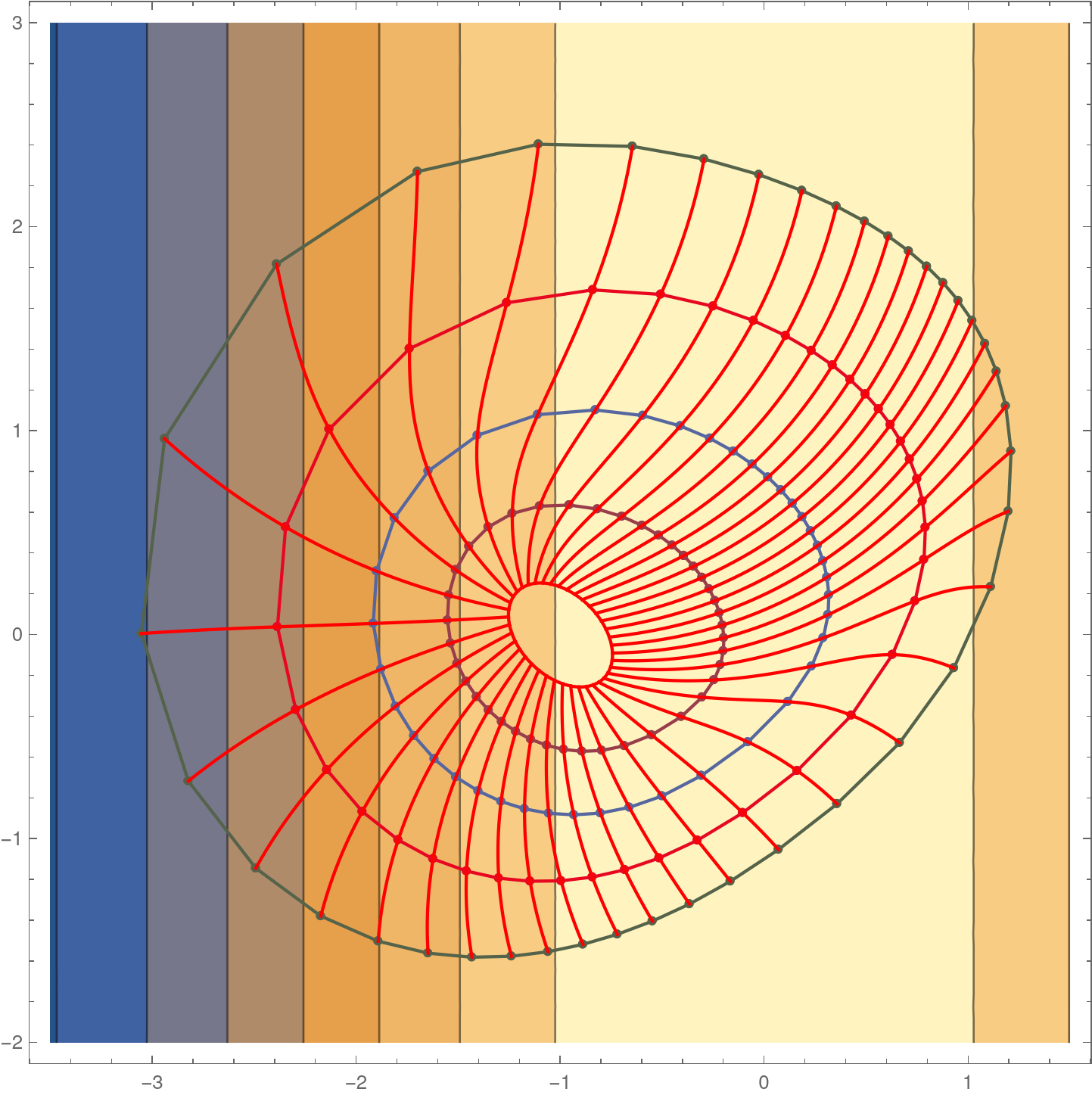}}
\subfigure{\label{fig:complex_3d}\includegraphics[width=0.49\textwidth]{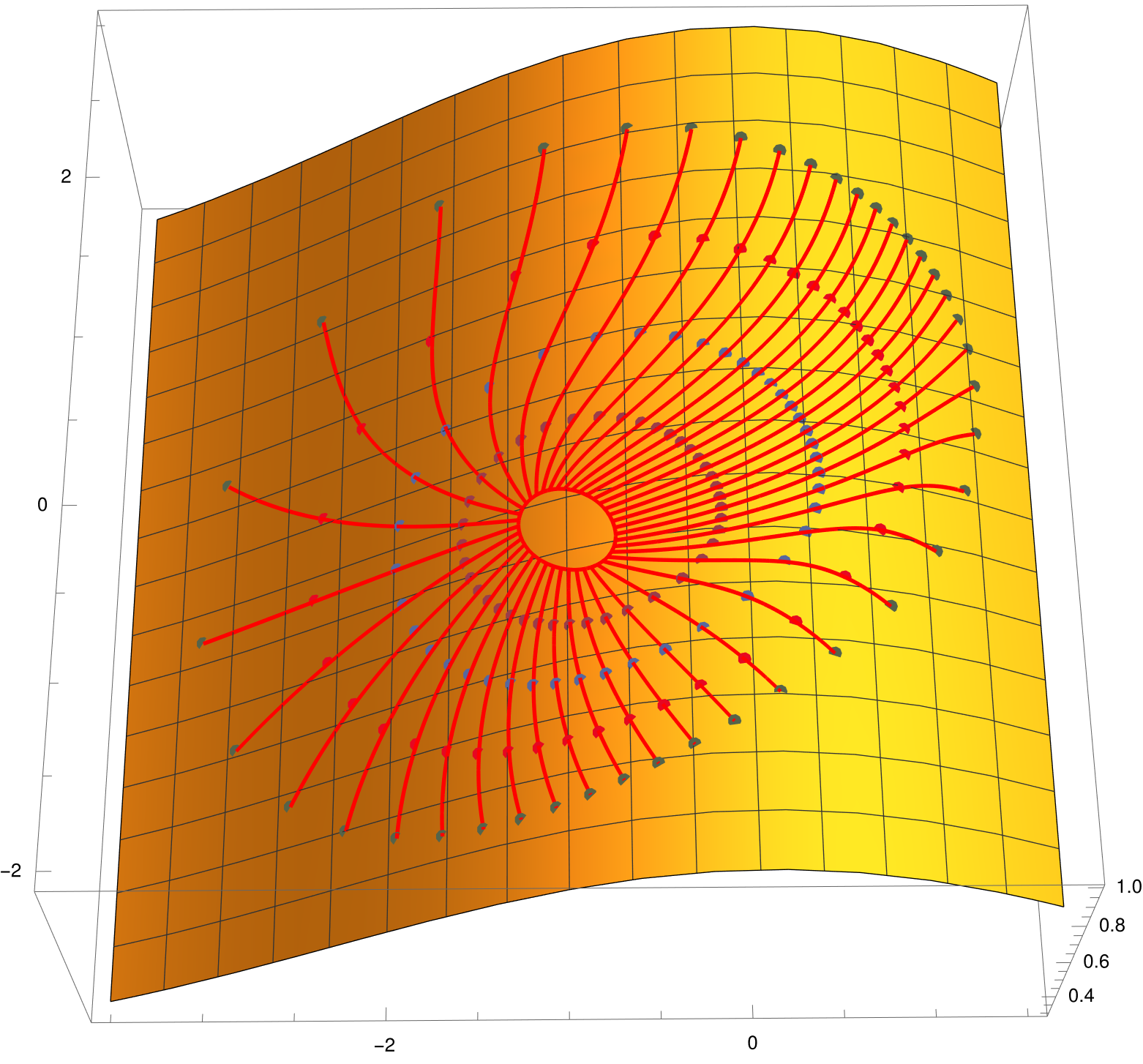}}
\caption{Example where both the wind and the fuel conditions do not remain constant. On the one hand, the wind becomes stronger and varies its direction over time ($ a $ and $ \tilde\phi $ increases). On the other hand, $ h $ decreases with the height of the mountain, representing that the vegetation becomes sparser at the top. Data: $ S_0: [0,2\pi] \ni s\mapsto (0.2\cos(s)\cos(\pi/4)-0.3\sin(s)\sin(\pi/4)-1,0.2\cos(s)\sin(\pi/4)+0.3\sin(s)\cos(\pi/4)) $, $ z(x,y)=\exp(-x^2/10) $, $ a=1+t $, $ h=1+x^2/2 $, $ \varepsilon=0.8 $, $ \tilde\phi=2t $ and $ \Delta t=0.7 $.}
\label{fig:complex}
\end{figure}

\begin{figure}
\centering
\subfigure{\label{fig:complex2_2d}\includegraphics[width=0.49\textwidth]{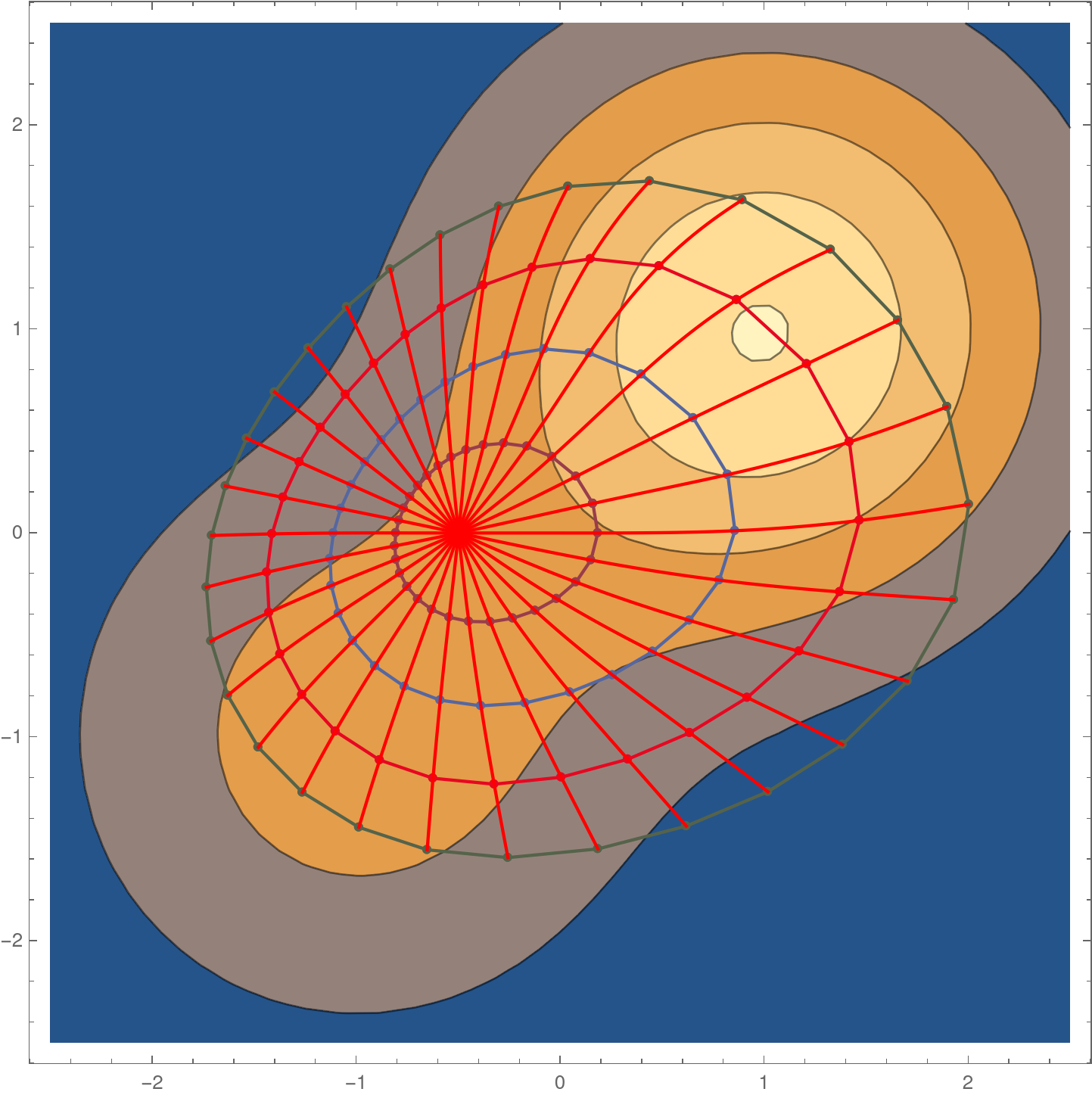}}
\subfigure{\label{fig:complex2_3d}\includegraphics[width=0.49\textwidth]{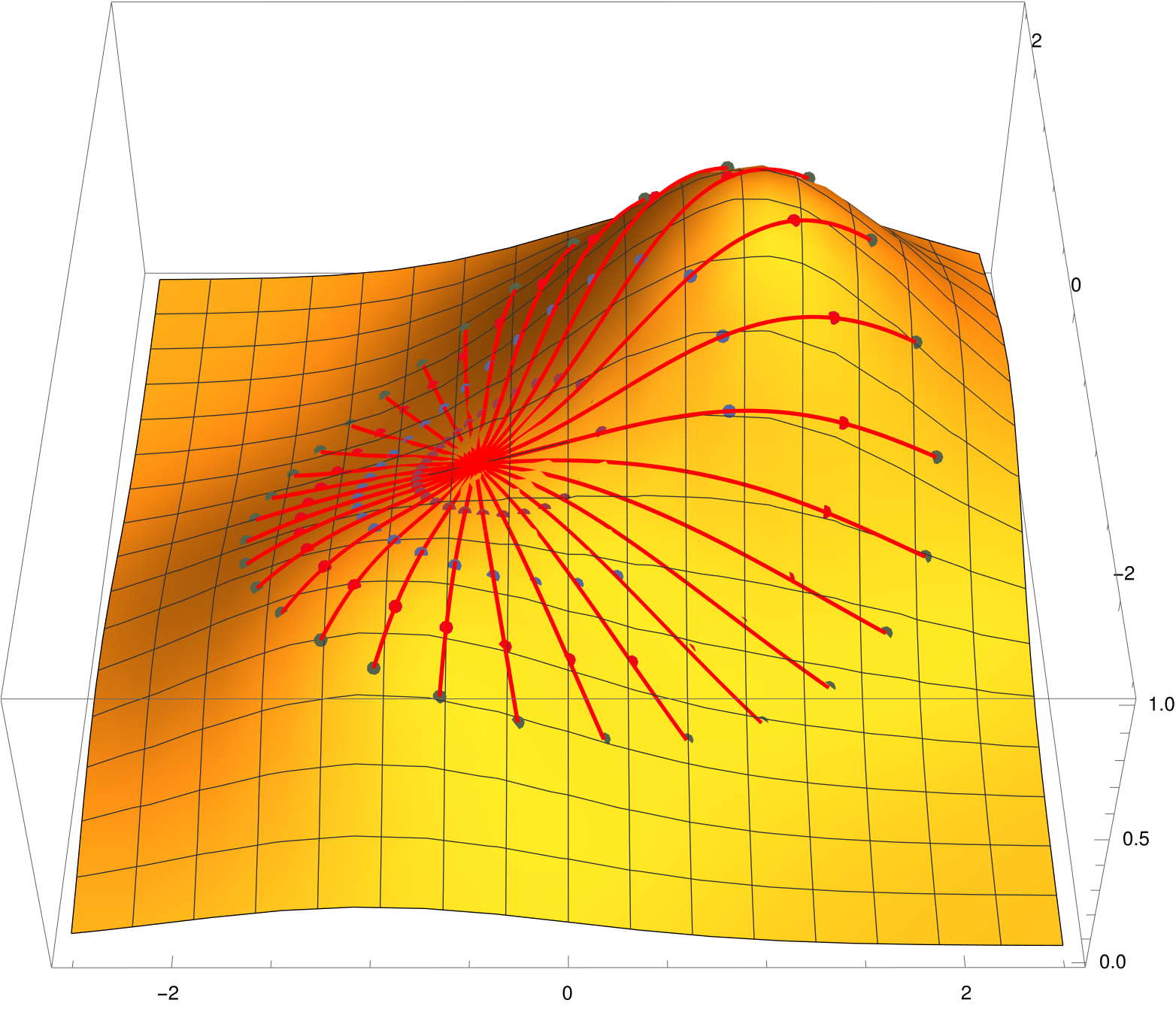}}
\caption{Example where the fire spreads over a more complex surface, with the wind blowing eastwards. As expected, the fastest trajectories are those heading northeast, i.e., downwind and upwards. Data: $ S_0=(-0.5,0) $, $ z(x,y)=\exp(-(x-1)^2/2-(y-1)^2/2)+1/2\exp(-(x+1)^2/2-(y+1)^2/2) $, $ a=h=1 $, $ \varepsilon=0.6 $, $ \tilde\phi=0 $ and $ \Delta t=1 $.}
\label{fig:complex2}
\end{figure}

\appendix

\section{Further issues}\label{appendix}
Here we will study in more depth two relevant technical issues. First, the role of the cut function and the different types of cut points. And secondly, the strong convexity hypothesis implicit in the Finslerian approach (condition (c) in Def.~\ref{def:finsler}(i)), 
as our model can be adapted 
when this hypothesis is not imposed a priori.
\subsection{Cut points}\label{appendix1}
Let us analyze the cut function introduced in \eqref{eq:cut}. Consider the spacetime trajectory $ \hat \gamma:t \rightarrow \hat f(t,s_0) $. As the wildfire is limited to a bounded (compact) region,\footnote{Technically, the required hypothesis is that the spacetime $ (M,G) $ is globally hyperbolic.} this  implies that a cut point $ \hat f(t_0,s_0)=(t_0,p_0) $ will only appear when either a second spacetime minimizing trajectory arrives at $(t_0,p_0)$, or there are spacetime trajectories starting close to the direction $\partial_t\hat f(0,s_0)$ and arriving at $\hat f(t,s_0)$ for $t (>t_0)$ close to $t_0$. More precisely:

\begin{prop}\label{prop:appendix}
Let $ (t_0,p_0)\in M $ be the cut point of $ \hat\gamma:t \mapsto \hat f(t,s_0) $. Then, either (a)
$ (t_0,p_0) $ is the first intersection point of $ \hat\gamma $ with another spacetime trajectory of the fire, or
(b) $ (t_0,p_0) $ is the first focal point of $ S_0 $ along $ \hat\gamma $
(it is possible for both conditions to hold simultaneously).
\end{prop}
\begin{proof}
If $ \hat\gamma(t_0)=(t_0,p_0) $ is the cut point of $ \hat\gamma $, by definition $ \hat\gamma(t_0) \in \partial J^+(B_0) $ and $ \hat\gamma(t_0+\varepsilon) \notin \partial J^+(B_0) $ for all $ \varepsilon > 0 $. So, given that our spacetime $ (M,G) $ is globally hyperbolic, we can construct a sequence of timelike geodesics $ \{\hat\varphi_\varepsilon\} $, each $ \hat\varphi_\varepsilon\ $ departing orthogonally from $ S_0 $ and arriving at $ \hat\gamma(t_0+\varepsilon) $. When normalized with respect to any Riemannian metric $ g $ on $ M $, the initial velocities of $ \{\hat\varphi_\varepsilon\} $ lie in the $ g $-unit tangent bundle of $ S_0 $, which is compact (due to the compactness of $ S_0 $). Thus, they converge to a causal velocity whose corresponding causal geodesic $ \hat\varphi $ arrives at $ (t_0,p_0)\in \partial J^+(B_0) $. By \cite[Thm. 6.9]{AJ}, $ \hat\varphi $ is a lightlike geodesic $ G $-orthogonal to $ S_0 $ with its inital velocity pointing outwards, i.e., it is a spacetime trajectory of the fire. There are two possibilities:
\begin{itemize}
\item If $ \hat\gamma \not= \hat\varphi $, then $ (t_0,p_0) $ is the intersection point of two spacetime trajectories of the fire, necessarily the first one. Indeed, assume $ \hat\gamma $ intersects another spacetime trajectory $ \hat\alpha $ at $ \tau < t_0 $. Then we can define, for any $ \varepsilon > 0 $, the curve
\begin{equation*}
\hat \rho(t) :=
\left\lbrace 
\begin{array}{l}
\hat \alpha(t), \quad t\in[0,\tau], \\
\hat \gamma(t), \quad t\in[\tau,\tau+\varepsilon],
\end{array}
\right.
\end{equation*}
which is a non-smooth causal curve from $ S_0 $ and thus, by \cite[Prop. 6.5]{AJ}, there exists a timelike curve from $ \hat\rho(0)\in S_0 $ to $ \hat\gamma(\tau+\varepsilon) $ for every $ \varepsilon > 0 $. We conclude that $ \hat\gamma(t_0) \notin \partial J^+(B_0) $, which is a contradiction.

\item If $ \hat\gamma = \hat\varphi $, then we can take a normal geodesic $ \hat\varphi_{\varepsilon} $ which intersects $ \hat\gamma $ at $ \hat\gamma(t_0+\varepsilon) $ and whose initial velocity is arbitrarily close to that of $ \hat\gamma $, i.e., the normal exponential map is not a local diffeomorphism around $ t_0\hat\gamma'(0) $, which implies that it is singular there (namely, the differential map cannot be an isomorphism at $ t_0\hat\gamma'(0) $) and, therefore, $ \hat\gamma(t_0)=(t_0,p_0) $ is a focal point of $ S_0 $ along $ \hat\gamma $ (see \cite[Prop. 10.30]{O}, which can be directly translated to the Lorentz-Finsler case), necessarily the first one (as the normal exponential map in the direction $ \hat\gamma'(0) $ is a local diffeomorphism before $ t_0 $).
\end{itemize}
\end{proof}

According to this proposition, cut points are classified into two groups: first focal points, which are meeting (or ``almost meeting'') points for different geodesics starting arbitrarily close to a point of the firefront, and non-focal ones, where two first-arriving fire trajectories coming from different directions meet. Typically, a focal point will have non-focal cut point neighbors. Eventually, the projection of the latter on the space might yield a curve $ \gamma $ starting at the projection of the focal point that can be used as an ``escape curve" by firefighters (see for example the straight line going through the top of the mountain in Fig.~\ref{fig:cut_points_sym}, starting at the first crossing of nearby fire trajectories, which would correspond to the first focal point). Indeed, $ \gamma $ would provide the safest way out a priori; if it can be parametrized by time at the crossing instant, it would represent the last chance to escape.



The type of cut point also becomes relevant when solving the PDE system \eqref{eq:ort_cond_F}. Indeed, at focal points $ \partial_sf $ vanishes and the system cannot be solved,\footnote{Note in Thm.~\ref{th:pde_ode} that we can only guarantee that the PDE system determines the firemap until $ t=\epsilon $, while the ODE system provides the firemap for any $ t \geq 0 $.} so the firefront has to be redefined as a regular curve. Otherwise, the PDE system keeps providing a solution beyond non-focal cut points but corrections to the firefront are still needed in order to distinguish burned from unburned regions. In any case though, these adjustments require much more time and computing power than in the ODE's case, where it is only a matter of removing the trajectories that reach their cut point, regardless of the type.

\subsection{On the convexity of the indicatrix}\label{appendix2}
As discussed in Rem.~\ref{rem:convex}, there might be cases where $ \Sigma $ fails to be strongly convex and, therefore, no longer defines a Finsler metric (usually when the slope is too steep and condition \eqref{eq:cond_no_wind} is not satisfied). Observe that in this situation, the triangle inequality does not hold, allowing situations such as the one depicted in Fig.~\ref{fig:non_str_convex}.

\begin{figure}
\centering
\subfigure[Curve $ \Sigma$ ($Q(v)=0 $) and its convex hull $ CH(\Sigma) $.]{\label{fig:convex_env}\includegraphics[width=0.49\textwidth]{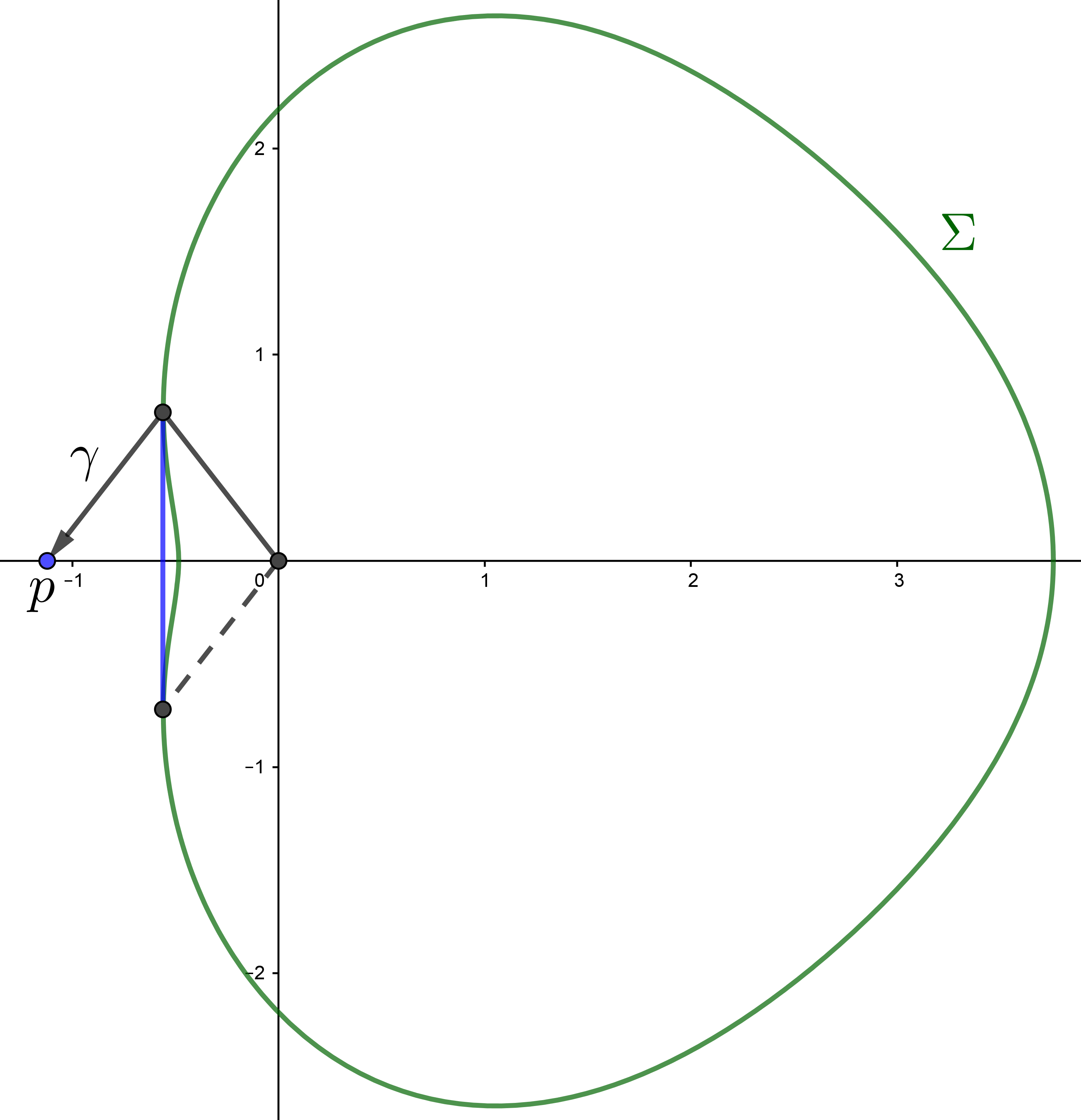}}
\subfigure[Detail of the part where $ \Sigma $ is not strongly convex.]{\label{fig:triangle_ineq}\includegraphics[width=0.49\textwidth]{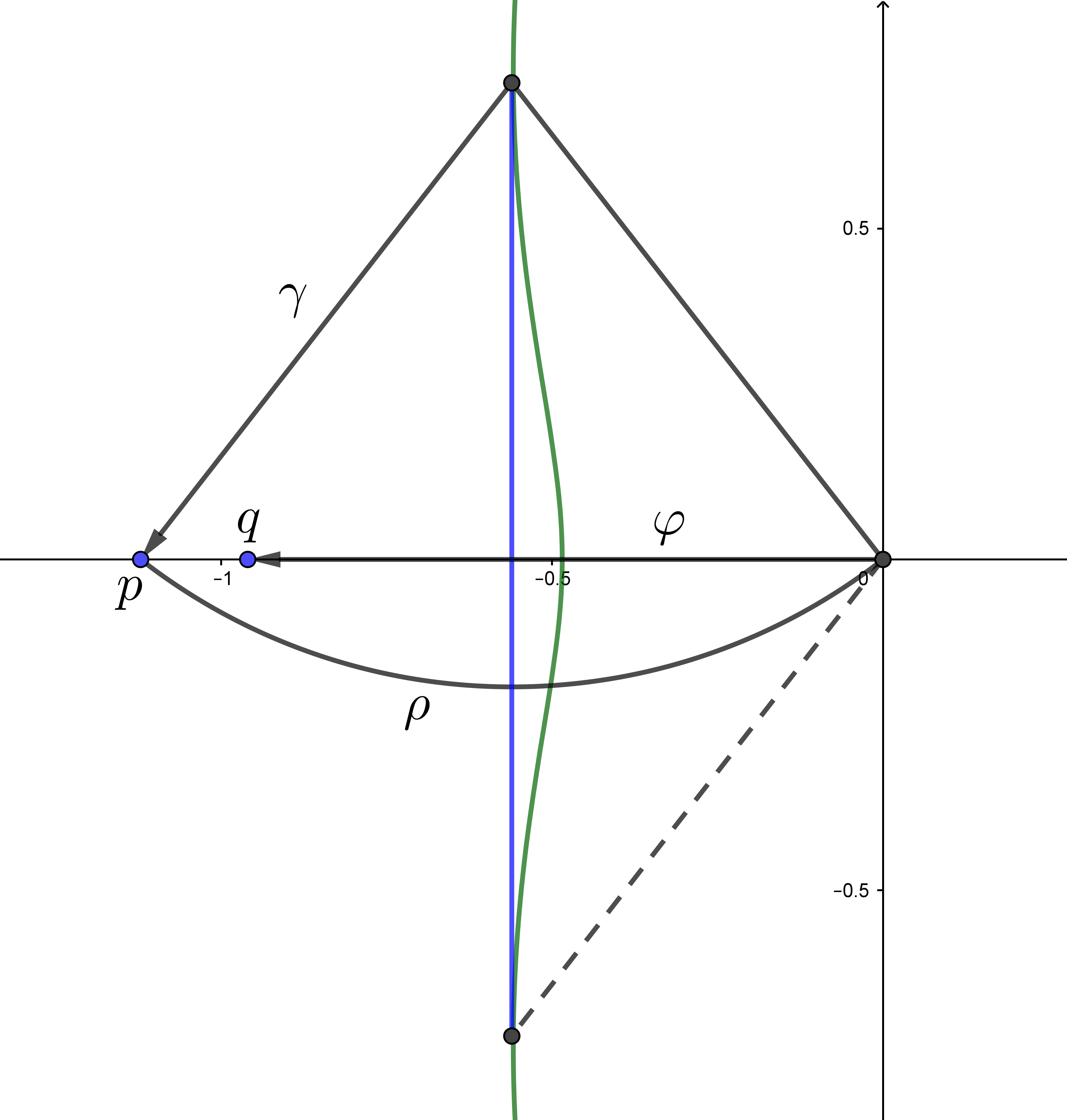}}
\caption{When $ \Sigma $ is not strongly convex, the triangle inequality no longer holds and the geodesics do not locally minimize the time. In the picture, $ \Sigma $ remains constant on $M$ and the broken line $\gamma$ spends two time units going from the origin to $ p $, while in that lapse, the straight segment $ \varphi $ only arrives at $ q $.  Taking the boundary of the convex hull of $ \Sigma $ as a new indicatrix, every curve whose velocity remains on the (non-strictly) convex part (the blue line in the figure) is a trajectory that reaches $ p $ in two time units, e.g., $ \rho $. In the end, the convex hull generates the same firefront as the original $ \Sigma $. Data: $ \partial_xz=1 $, $ \partial_yz = 0 $, $ a=1 $, $ h=2 $, $ \varepsilon=0.9 $ and $ \tilde{\phi}=0 $.}
\label{fig:non_str_convex}
\end{figure}

To pose this issue in simple terms, assume first that the velocity of the fire is given by some $ \Sigma $ independent of the point in $M=\R\times N$, but we do not assume the convexity of $\Sigma$. Then, after one time unit the firefront will not coincide with $ \Sigma $, but with the boundary of its convex hull, $ \partial CH(\Sigma) $. This shows that the firefront can be effectively computed by taking $ \partial CH(\Sigma) $ as the new indicatrix (recall the last part of Rem.~\ref{rem:infinit_fire}). 

In general, $ \partial CH(\Sigma) $ is only locally Lipschitz, although in our particular model it is smooth everywhere except at the tangency points of the additional segments (the blue one in Fig.~\ref{fig:convex_env}), where it is  $ C^1 $. This will define pointwise a (non-symmetric) norm, which  it is not exactly a Minkowski norm because it is not smooth everywhere and it may degenerate. Specifically, the fundamental tensor in \eqref{eq:fund_tensor} is not defined in the directions where $ \partial CH(\Sigma) $ is not $ C^2 $, and it is degenerate in the directions where $ \partial CH(\Sigma) $ lies on a straight line.

Working with this convex hull, one arrives at a possibly degenerate and non-smooth time-dependent Finsler metric $F^0$ and the corresponding Lorentz-Finsler one $G^0=dt^2-(F^0)^2$. This has clear advantages over working directly with $ \Sigma $: the picture of the wildfire trajectories is simplified, the regularity of the model is improved and it is consistent with the purpose of this work (and the applicability of the wildfire models in general), in which the interest is focused on the firefront rather than on the trajectories of the ``fire particles''.

However, even when working with $G^0$ the computation of the firefront is somewhat tricky, because the uniqueness theorem for its geodesics does not hold. Indeed, Fig.~\ref{fig:non_str_convex} shows a strong lack of uniqueness even in the case when $ F^0 $ is a norm (independent of $(t,p)$): every curve $ \gamma(t) $ with velocity in the non-strongly convex part of $ \partial CH(\Sigma) $ becomes a geodesic (up to reparametrizations) for $ F^0 $ and, thus, $ \hat\gamma(t)=(t,\gamma(t)) $ becomes a cone geodesic. So, there might be many possible trajectories of the fire from a given point.

From a practical viewpoint, we need to break this degeneracy in order to effectively compute the firefront. In some particular cases this can be carried out in a simple systematic way. For instance, if $ F^0 $ is independent of $t$ and Berwald (in particular, when $ F^0 $ is a norm),\footnote{A (non necessarily regular) Finsler metric is Berwald when the indicatrices at different points are affinely equivalent. There are several characterizations of this property, see \cite{Mat} and references therein.} there exists a Riemannian metric on $ N $ whose geodesics are also $ F^0 $-geodesics (see \cite[Thm. 1]{Mat}). In this case, the firefront can be obtained using this Riemannian metric. This might be extended to the time-dependent case when compatibility with a time-dependent Riemannian metric holds. In general though, a way to circumvent the problem would be to approximate $ \partial CH(\Sigma) $ by a strongly convex indicatrix. From a practical viewpoint, finding an appropriate accuracy in the approximation should not be a major problem. However, this raises theoretical questions on the convergence of different approximations to be studied elsewhere.

\section*{Acknowledgments}
MAJ and EPR were partially supported by the projects PGC2018-097046-B-I00 and PID2021-124157NB-I00, funded by MCIN/AEI/10.13039/501100 011033/ "ERDF A way of making Europe", and also by Ayudas a proyectos para el desarrollo de investigaci\'{o}n cient\'{i}fica y t\'{e}cnica por grupos competitivos (Comunidad Aut\'{o}noma de la Regi\'{o}n de Murcia), included in the Programa Regional de Fomento de la Investigaci\'{o}n Cient\'{i}fica y T\'{e}cnica (Plan de Actuaci\'{o}n 2022) of the Fundaci\'{o}n S\'{e}neca-Agencia de Ciencia y Tecnolog\'{i}a de la Regi\'{o}n de Murcia, REF. 21899/PI/22. EPR and MS were partially supported by the project PID2020-116126GB-I00 funded by MCIN/AEI/10.13039/501100011033 and P20-01391 (PAIDI 2020, Junta de Andaluc\'{i}a), as well as the framework IMAG-Mar\'{i}a de Maeztu grant CEX 2020-001105-M/AEI/10.13039/501100011033. EPR was also supported by Ayudas para la Formaci\'{o}n de Profesorado Universitario (FPU) from the Spanish Government.

\end{document}